%% file: gaussian_law_errors.tex
\documentclass[12pt]{amsart}
\usepackage[utf8]{inputenc}
\usepackage{amsfonts}
\usepackage{latexsym}
\usepackage{amssymb}
\usepackage{amsmath}
\usepackage{color}
\usepackage{tikz}
\usepackage{enumerate}
\usepackage{mathrsfs}
\usepackage{todonotes}
\usepackage{hyperref}

\usepackage[left=2.5cm, top=2.5cm,bottom=2.5cm,right=2.5cm]{geometry}

\newcommand{\R}{\mathbb R}
\newcommand{\N}{\mathbb N}

\newcommand{\E}{\mathbb E}
\newcommand{\Q}{\mathbb Q}
\newcommand{\Z}{\mathbb Z}

\newcommand{\Pro}{\mathbb P}

\newcommand{\Var}{\mathrm{Var}}

\newcommand{\Pc}{\mathcal{P}}

\def\dint{\textup{d}}

\newcommand{\ind}{{\small 1}\!\!1}
\newcommand{\sign}{\text{sign}}
\newcommand{\frm}{finitely measurable}
\newcommand{\sod}{sum-of-digits function}

\DeclareMathOperator{\id}{id}

\newtheorem{thm}{Theorem}[section]

\newtheorem{cor}[thm]{Corollary}
\newtheorem{lemma}[thm]{Lemma}
\newtheorem{df}[thm]{Definition}
\newtheorem{proposition}[thm]{Proposition}

{
\theoremstyle{definition}
\newtheorem{example}[thm]{Example}
\newtheorem{rem}[thm]{Remark}

}

\definecolor{gurot}{RGB}{180,20,20}
\definecolor{jorot}{RGB}{220,20,20}

\usepackage{nomencl}
\makenomenclature

\begin{document}


\title[]{Independence in Mathematics -- the key \\to a Gaussian law}

\author[G. Leobacher]{Gunther Leobacher}
\address{Gunther Leobacher: Institute of Mathematics \& Scientific Computing, University of Graz, Austria.}
\email{gunther.leobacher@uni-graz.at}

\author[J. Prochno]{Joscha Prochno}
\address{Joscha Prochno: Institute of Mathematics \& Scientific Computing, University of Graz, Austria.} \email{joscha.prochno@uni-graz.at}

\keywords{Central limit theorem, Gaussian law, independence, relative measure, lacunary series}
\subjclass[2010]{Primary: 60F05, 60G50 Secondary: 42A55, 42A61}



\begin{abstract}
In this manuscript we discuss the notion of (statistical) independence embedded
in its historical context. We focus in particular on its appearance and role in
number theory, concomitantly exploring the intimate connection of independence
and the famous Gaussian law of errors. As we shall see, this at times requires
us to go adrift from the celebrated Kolmogorov axioms, which give the
appearance of being ultimate ever since they have been introduced in the
$1930$s. While these insights are known to many a mathematician, we feel it is
time for both a reminder and renewed awareness. Among other things, we present the independence of the coefficients in a binary expansion together with a central limit theorem for the \sod{} as well as the independence of divisibility by
primes and the resulting, famous central limit theorem of Paul Erd\H{o}s and
Mark Kac on the number of different prime factors of a number $n\in\N$. We
shall also present some of the (modern) developments in the framework of
lacunary series that have its origin in a work of Rapha\"el Salem and Antoni
Zygmund.
\end{abstract}

\maketitle


\section{Introduction}

One of the most famous graphs, not only among mathematicians and scientists, is the probability density function of the (standard) normal distribution (see Figure~\ref{fig:glockenkurve}), which has adorned the $10$ Mark note of the former German currency for many years. Although already taking a central role in a work of Abraham de
Moivre (26. May 1667 in Vitry-le-Francois; 27. November 1754 in London) from
$1718$, this curve only earned its enduring fame through the work of famous German
mathematician Carl Friedrich Gau{\ss} (30. April 1777 in Braunschweig; 23.
February 1855 in Göttingen), who used it in the approximation of orbits by
ellipsoids when developing the least squares method, nowadays a standard
approach in regression analysis. More precisely, Gau{\ss} conceived this method
to master the random errors, i.e., those which fluctuate due to the
unpredictability or uncertainty inherent in the measuring process, that occur
when one tries to measure orbits of celestial bodies. The strength of this
method became apparent when he used it to predict the future location of the
newly discovered asteroid Ceres. Ever since, this curve seems to be the key to
the mysterious world of chance and still the myth holds on that wherever this
curve appears, randomness is at play.

With this article we seek to address mathematicians as well as a mathematically educated audience alike. One can say that the goal of this manuscript is $3$-fold. First, for those less familiar with it we want to undo the fetters that connect chance and the Gaussian curve so onesidedly. Second, we want to recall the deep and intimate connection of the notion of statistical independence and the Gaussian law of errors beyond classical probability theory, which, thirdly, demonstrates that occasionally one is obliged to step aside from its seemingly ultimate form in terms of the Kolmogorov axioms and work with notions having its roots in earlier foundations of probability theory.

To achieve this goal we shall, partially embedded in a historic context,
present and discuss several results from mathematics where, once an appropriate
form of statistical independence has been established, the Gaussian curve
emerges naturally. In more modern language this means that central limit
theorems describe the fluctuations of mathematical quantities in different
contexts. Our focus shall be on results that nowadays are considered to be part
of probabilistic number theory. At the very heart of this development lies the
true comprehension and appreciation of independence by Polish mathematician
Mark Kac (3. August 1914 in Kremenez; 26. October 1984 in California). His
pioneering works and insights, especially his collaboration with Hugo Steinhaus
(14. January 1887 in Jas{\l}o; 25.  February 1972 in
Wroc{\l}aw) and famous mathematician Paul Erd\H{o}s (26. March 1913 in
Budapest; 20. September 1996 in Warsaw), have revolutionized our understanding
and formed the development of probabilistic number theory for many years with
lasting influence.

\begin{center}


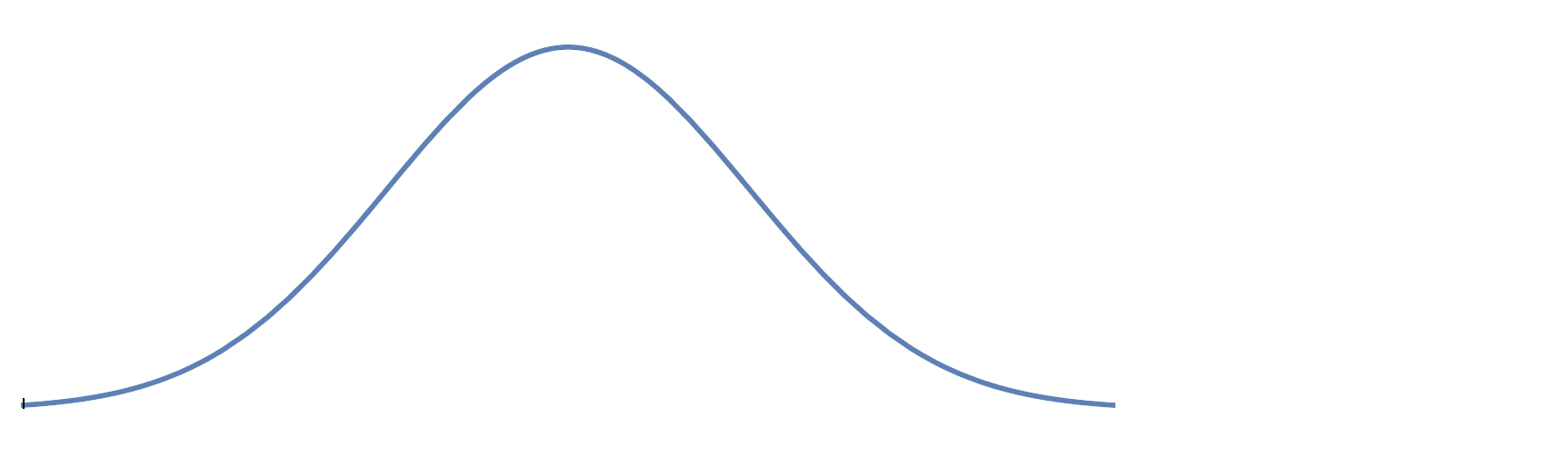

\label{fig:glockenkurve}

Figure~\ref{fig:glockenkurve}. The Gaussian curve.
\end{center}

\section{The classical central limit theorems and independence -- a refresher}

In this section we start with two fundamental results of probability theory and the notion of independence. These considerations form the starting point for future deliberations.

\subsection{The notion of independence}

Independence is one of the central notions in probability theory. It is hard to imagine today that this, for us so seemingly elementary and simple concept, has only been used vaguely and intuitively for hundreds of years without a formal definition underlying this notion. Implicitly this concept can be traced back to the works of Jakob Bernoulli (6. January 1655 in Basel; 16. August 1705 in Basel) and evolved in the capable hands of Abraham de Moivre. In his famous oeuvre ``The Doctrine of Chances'' \cite{Moivre1718} he wrote:
\vskip 2mm
``\emph{...if a Fraction expresses the Probability of an Event, and another Fraction the Probability of another Event, and those two Events are independent; the Probability that both those Events will Happen, will be the Product of those Fractions.}'' 
\vskip 2mm
It is to be noted that, even though this definition matches the modern one,
neither the notion ``Probability'' nor ``Event'' had been introduced in an
axiomatic way. It seems that the first formal definition of independence goes
back to the year $1900$ and the work \cite{Bohlmann1900} of German
mathematician Georg Bohlmann  (23. April 1869 in Berlin; 25. April 1928 in
Berlin)\footnote{It was decades later that Hugo Steinhaus and Mark Kac
rediscovered this concept independently of the other mathematicians
\cite{KacSteinhaus1938}. They were unaware of the previous works.}. In fact,
long before Andrei Nikolajewitsch Kolmogorov (25. April 1903 in Tambow; 20.
October 1987 in Moscow) proposed his axioms that today form the foundation of
probability theory, Bohlmann had presented an axiomatization --- but without asking
for $\sigma$-additivity. For a detailed exposition of the historical
development and the work of Bohlmann, we refer the reader to an article of
Ulrich Krengel \cite{Krengel2011}.

We continue with the formal definition of independence as it is used today. Let $(\Omega,\mathcal A, \Pro)$ be a probability space consisting of a non-empty set $\Omega$ (the sample space), a $\sigma$-Algebra (the set of events) on $\Omega$, and a probability measure $\Pro: \mathcal A \to [0,1]$. We then say that two events $A,B\in\mathcal A$ are \emph{(statistically) independent} if and only if
\[
\Pro[A\cap B] = \Pro[A]\cdot \Pro[B]\,.
\]
In other words two events are independent if their joint probability equals the product of their probabilities. This naturally extends to any sequence $A_1,A_2,\dots$ of events, which are said to be independent if and only if for every $n\in\N$ and all subsets $I\subseteq \N$ of cardinality $n$,
\[
\Pro\Big[\bigcap_{i\in I}A_i\Big] = \prod_{i\in I}\Pro[A_i]\,.
\]
It is important to note that in this case we ask for much more than just pairwise independence and, consequently, also have to verify much more. The number of conditions to be verified to show that $n$ given events are independent is exactly
\[
{n\choose2} + {n\choose 3} + \dots + {n\choose n} = 2^n - (n+1). 
\]
Having this notion of independence at hand, we define independent
random variables. If $X:\Omega\to \R$ and $Y:\Omega\to \R$ are two random
variables, then we say they are independent if and only if for all measurable
subsets $A,B\subseteq \R$, \[
\Pro[X\in A, Y\in B] = \Pro[X\in A]\cdot \Pro[Y\in B]\,.
\footnote{We use the standard notation $\{X\in A\}$ for $\{\omega\in\Omega\,:\, X(\omega)\in A\}$,  $\Pro[X\in A]$ for $\Pro[\{X\in A\}]$, and $\Pro[X\in A, Y\in B]$ for $\Pro[\{X\in A\}\cap\{Y\in B\}]$.}
\]
This means that the random variables $X$ and $Y$ are independent if and only if for all measurable subsets $A,B\subseteq \R$ the events $\{X\in A \}\in \mathcal A$ and $\{Y\in B\}\in\mathcal A$ are independent. Again, a sequence $X_1, X_2,\dots:\Omega\to\R$ of random variables is said to be independent if and only if for every $n\in\N$, any subset $I\subseteq \N$ of cardinality $n$, and all measurable sets $A_i\subseteq \R$, $i\in I$,
\[
\Pro\bigg[ \bigcap_{i\in I}\{X_i \in A_i \}\bigg] = \prod_{i\in I} \Pro[X_i\in A_i]\,.
\]


\subsection{The central limit theorems of de Moivre-Laplace and Lindeberg}

The history of the central limit theorem starts with the work of French mathematician Abraham de Moivre, who, around the year $1730$, proved a central limit theorem for standardized sums of independent random variables following a symmetric Bernoulli distribution \cite{Moivre2edition}.\footnote{A random variable $X$ is Bernoulli distributed if and only if $\Pro(X=0)+\Pro(X=1)=1$. Here $p=\Pro(X=1)$ is the parameter of the Bernoulli distribution and in the case where $p=\frac{1}{2}$, we call the distribution ``symmetric''. In his paper de Moivre did not call them Bernoulli random variables, but spoke of the probability distribution of the number of heads in coin toss.} It was not before $1812$ that Pierre-Simon Laplace (28. March 1749 in Beaumont-en-Auge; 5. March 1827 in Paris) generalized this result to the asymmetric case \cite{Laplace1812}.
However, a central limit theorem for standardized sums of independent random variables together with a rigorous proof only appeared much later in a work of Russian mathematician Alexander Michailowitsch Ljapunov (06. June 1857 in Jaroslawl; 03. November 1918 in Odessa) from 1901 \cite{Ljapunov1901}. Jarl Waldemar Lindeberg (04. August 1876 in Helsinki; 12. December 1932 Helsinki) published his works on the central limit theorem, in which he developed his famous and ingenious method of proof (today known as Lindeberg method), in 1922 \cite{L1922a,L1922b}. While in a certain sense elementary, this technique can be applied in various ways. A very nice exposition on Lindeberg's method can be found in the survey article \cite{EL2014} of Peter Eichelsbacher and Matthias L\"owe. For an exhaustive presentation on the history of the central limit theorem we warmly recommend the monograph of Hans Fischer \cite{Fischer2011}.

Let us start with the classical central limit theorem of de Moivre, hence restricting ourselves to the symmetric case $p=\frac{1}{2}$ in the Bernoulli distribution. 

\begin{thm}[De Moivre, 1730]\label{clt:de moivre}
Let $X_1,X_2,X_3,\dots$ be a sequence of independent random variables with a symmetric Bernoulli distribution. Then, for all $a,b\in\R$ with $a<b$, we have
\[
\lim_{n\to\infty} \Pro\bigg[a \leq \frac{\sum_{k=1}^nX_k - \frac{n}{2}}{\sqrt{\frac{n}{4}}} \leq b\bigg] = \frac{1}{\sqrt{2\pi}}\int_a^b e^{-\frac{x^2}{2}}\,\dint x\,.
\]
\end{thm}

The theorem of de Moivre, when discussed in school for instance, can be nicely depicted using the Galton Board (also known as bean machine). 
Let us consider the experiment of throwing an ideal and fair coin $n$-times (i.e., head shows up with probability $1/2$). The single throws are regarded to be independent as none of them influences the other. The number $k$ of heads showing up in that experiment is a number between $0$ and $n$. The probability that we see heads exactly $k$-times is described by a binomial distribution. Now de Moivre's theorem says that, for a large number $n$ of tosses tending to infinity, the form of the discrete distribution function approaches the Gaussian curve. 

We have already mentioned at the beginning of this section that under suitable conditions a central limit theorem for general independent random variables may be obtained, not only those describing or modeling a coin toss. 

We formulate Lindeberg's central limit theorem. In what follows, we shall denote by $\ind_A$ the indicator function of the set $A$, i.e., $\ind_A(x)\in \{0,1\}$ with $\ind_A(x)=1$ if and only if $x\in A$. The {\em expectation} of a random 
variable $X$ with respect to the probability measure $\Pro$ is defined as
$\E[X]:=\int_\Omega X \dint \Pro$, if this integral is defined. $X$  is called
{\em centered} if and only if $\E[X]=0$. If $\E[|X|]<\infty$ we define
the variance by $\Var[X]:=\E\big[(X-\E[X])^2\big]$.

\begin{thm}[Lindeberg CLT, 1922] \label{thm:lindeberg}
Let $X_1,X_2,X_3,\dots$ be a sequence of independent, centered, and square integrable random variables. Assume that for each $\varepsilon\in(0,\infty)$,
\[
L_n(\varepsilon) := \frac{1}{s_n^2} \sum_{k=1}^n \E\big[X_k^2\,\ind_{\{|X_k| > \varepsilon s_n\}}\big] \,\,\stackrel{n\to\infty}{\longrightarrow}\,\, 0 \qquad(\textrm{Lindeberg condition}),
\] 
where $s_n^2:= \sum_{k=1}^n \Var[X_k]$. Then, for all $a,b\in\R$ with $a<b$, we have
\[
\lim_{n\to\infty}\Pro\bigg[ a\leq \frac{\sum_{k=1}^n X_k }{s_n} \leq b \bigg] = \frac{1}{\sqrt{2\pi}}\int_a^b e^{-\frac{x^2}{2}}\,\dint x\,.
\]
\end{thm}

Lindeberg's condition guarantees that no single random variable has too much influence. This immediately becomes apparent when looking at the Feller condition, which is implied by Lindeberg's condition. We refrain from discussing or presenting the details and refer again to \cite{EL2014}.

\begin{rem}
Let us assume that the random variables in Theorem \ref{thm:lindeberg} are identically distributed and have variance $\Var[X_k]=\sigma^2\in(0,\infty)$ for all $k\in\N$. Then Lindeberg's condition is automatically satisfied: 
\[
s_n^2 = \sum_{k=1}^n\Var[X_k] = n\sigma^2
\]
and therefore, since the random variables $X_k$ are identically distributed, we obtain for any $\varepsilon>0$ that
\begin{align*}
L_n(\varepsilon) & = \frac{1}{n\sigma^2} \sum_{k=1}^n \E\Big[ X_k^2\ind_{\{|X_k|>\varepsilon s_n \}}\Big]
 = \frac{1}{n\sigma^2} \sum_{k=1}^n \E\Big[ X_1^2\,\ind_{\{|X_1|>\varepsilon \sqrt{n} \sigma \}}\Big] \cr
 & = \frac{1}{\sigma^2} \E\Big[ X_1^2\,\ind_{\{|X_1|>\varepsilon \sqrt{n} \sigma \}}\Big] \,\,\stackrel{n\to\infty}{\longrightarrow}\,\, 0,
\end{align*}
where the convergence to $0$ is a consequence of the Beppo Levi Theorem\footnote{Which is a version of the monotone convergence theorem.}.
\end{rem}

The previous remark immediately implies the classical central limit theorem for independent and identically distributed random variables. 

\begin{cor}\label{cor: clt iid}
Let $X_1,X_2,X_3,\dots$ be a sequence of independent and identically distributed random variables with $\E[X_1]=0$ and $\Var[X_1]=\sigma^2\in(0,\infty)$. Then, for all $a,b\in\R$ with $a< b$, we have
\[
\lim_{n\to\infty}\Pro\bigg[ a\leq \frac{\sum_{k=1}^n X_k }{\sqrt{n\sigma^2}} \leq b \bigg] = \frac{1}{\sqrt{2\pi}}\int_a^b e^{-\frac{x^2}{2}}\,\dint x\,.
\]
\end{cor}

One thing we immediately notice in the general version of Lindeberg's central limit theorem is the universality towards the underlying distribution of the random variables. Hence, the distribution seems to be irrelevant. On the other hand, in both the central limit theorem of de Moivre and the one of Lindeberg, we require the random variables to be independent. Could it be that independence is the key to a Gaussian law of errors? If so, does this connection go deeper and beyond a purely probabilistic framework? In the remaining parts of this work we want to get to the bottom of those questions.

\subsection{Binary expansion and independence}\label{sec:dyadisch}

In this section we will present a first example which a priori is non probabilistic. It has to do with intervals corresponding to binary expansions of real numbers $x\in[0,1]$ and a corresponding product rule for their lengths.

For simplicity, we start by reminding the reader of the decimal expansion of a number $x\in[0,1)$. One can prove that each number $x\in[0,1)$ has a non-terminating and unique decimal expansion (see, e.g., \cite{Behrends2009}). For example,
\[
\frac{2}{7} = 0{,}285714285714\ldots
\]
and this expression is merely a short way for writing
\[
\frac{2}{7} = \frac{2}{10} + \frac{8}{10^2} + \frac{5}{10^3} + \frac{7}{10^4}+\dots\,.
\]
Generally, for each $x\in[0,1)$ there exist unique numbers $d_1(x),d_2(x),d_3(x),\ldots$ in $\{0,1,\ldots,9 \}$ such that
\[
x = \frac{d_1(x)}{10} + \frac{d_2(x)}{10^2} + \frac{d_2(x)}{10^3} + \ldots\,.
\]
Analogous to the decimal expansion, each number $x\in[0,1)$ has a binary expansion (also known as dyadic expansion), i.e., there are unique numbers $b_1(x),b_2(x),b_3(x),\ldots$ in the set $\{0,1 \}$ such that
\begin{equation}\label{eq:dyadische entwicklung}
x= \frac{b_1(x)}{2}+ \frac{b_2(x)}{2^2} + \frac{b_3(x)}{2^3} + \dots\,.
\end{equation}
For instance, we can write
\[
\frac{2}{7} = \frac{0}{2} + \frac{1}{2^2} + \frac{0}{2^3} + \frac{0}{2^4} + \frac{1}{2^5} + \frac{0}{2^6} + \dots\,.
\] 
To guarantee uniqueness in the expansion, we agree to write the expansion in such a way that infinitely many of the binary digits are zero. As already indicated by the way we write it, the binary digits are functions in the variable we denoted by $x$, i.e.,
\[
b_k:[0,1) \to\{0,1\}, \qquad x\mapsto b_k(x).
\]
Sometimes these functions are called Rademacher functions, although Hans Rademacher (3. April 1892 in Wandsbek; 7. February 1969 in Haverford) defined a slightly different version \cite{Rademacher1922}. The value that $b_k$ takes at $x$ not only provides information about the $k$-th binary digit of $x$, but also about $x$ itself. Obviously, if $b_1(x)=1$, then $x\in[1/2,1)$ or if $b_2(x) = 0$, then $x\in[0,1/4)\cup[1/2,3/4)$. More generally, if we define for each $k\in\N$ the set
\[
B_k:= \bigcup_{j=1}^{2^{k-1}} \Big[ \frac{2j-2}{2^k},\frac{2j-1}{2^k} \Big),
\]
then
\[
b_k(x) = 1_{[0,1)\setminus B_k}(x) = 
\begin{cases}
0 & : x\in B_k \\
1 & : x\in [0,1)\setminus B_k\,. 
\end{cases}
\]
These considerations yield the following: if $n\in\N$, $k_1,\dots,k_n\in\N$, and $\varepsilon_1,\dots\varepsilon_n\in\{0,1\}$, then
\begin{align*}
\lambda\left( \bigcap_{i=1}^n b_{k_i}^{-1}(\varepsilon_i)\right) & =\lambda\big(\{x\in[0,1)\,:\,b_{k_1}=\varepsilon_1,\dots,b_{k_n}=\varepsilon_n \}\big) \\
& = \Big(\frac{1}{2}\Big)^n = \prod_{i=1}^{n} \lambda\big(\{x\in[0,1)\,:\,b_{k_i}=\varepsilon_i\big\}\big),
\end{align*}
where $\lambda$ denotes the $1$-dimensional Lebesgue measure (which in this
case simply assigns the length to an interval). This implies that the binary
coefficients as functions in $x\in[0,1)$, are independent; a result seemingly
discovered by French mathematician \'Emile Borel (7. January 1871 in
Saint-Affrique; 3. February 1956 in Paris) in 1909 \cite{Borel1909}. In
particular, the random variables $X_k=b_k$ satisfy the assumptions of de
Moivre's theorem (Theorem \ref{clt:de moivre}) and so we obtain a central limit
theorem for binary expansions $b_k$. Probability in the sense of coin tosses or
events has not played any role in our arguments. 
(Nevertheless, technically the $X_k$'s are
bona-fide random variables on the probability space 
$\big([0,1),\mathcal {B}([0,1)),\lambda)$.)

\subsection{Prime factors and independence}\label{subsec:primfaktoren-unabh}

We shall now consider a fundamentally different example of independence
in mathematics. Take a sufficiently large natural number $N\in \N$.
We note that roughly half of the numbers between $1$ and $N$ are divisible
by the prime number $2$, namely $2,4,6$ and so on. In the same way, roughly
one third of the numbers between $1$ and $N$ are divisible
by the prime number $3$, namely $3,6,9$ and so on. If we now consider 
the numbers between $1$ and $N$ which are divisible by $6$, then this is 
again roughly one sixth. However, divisibility by $6$ is equivalent to 
both divisibility by $2$ {\em and} $3$ and we can write this as
\[
\frac{1}{6} = \frac{1}{2}\cdot \frac{1}{3}
\]
for the corresponding fractions of numbers between $1$ and $N$.  But this 
reminds us of the multiplication of probabilities --- as occurring in the
concept of independence! Of course, the same argument applies for 
divisibility by general distinct primes $p$ and $q$ as well as by
any finite number of primes. We can say, in this sense, that divisibility
of a number by distinct primes is independent.

Apparently, every second natural number is divisible by $2$, so that 
the numbers with this property constitute one half of all natural numbers.
One could thus think that a randomly chosen natural number is 
divisible by $2$ with probability  $\frac 1 2$. In the same way, 
this number would be divisible by $3$ with probability $\frac 1 3$, and
an analog statement would hold for divisibility by every natural number.

It turns out that this notion, although intuitive, is incompatible 
with Kolmogorov's concept of probability in that no probability measure
on the naturals with the above property exists.

To see this, define, for every pair of numbers
$n,k$ with 
$n\in\N$ and $k\in\{1,\dots,n \}$ the set 
$A_{n,k}:=\{jn+k\colon j\in\N\cup\{0\} \}$. 
For $k\ne n$,  $A_{n,k}$ consists of all natural numbers which yield
remainder $k$ after division by $n$,
while for $k=n$ we have $A_{n,k}=A_{n,n}$, which is the set of all natural numbers that are divisible by $n$. We denote by  $\Pc(\N)$ set of all subsets of $\N$.

\begin{lemma}\label{lem:nomeasure}
Let $\mu$ be a finite measure on the set $\Pc(\N)$, which satisfies 
\begin{equation}\label{eq:union-prime}
\mu(A_{p,k})=\mu(A_{p,p})
\end{equation}
for every prime number $p$ and every $k\in\{1,\dots,p \}$. 
Then $\mu(\{m\})=0$ for every
$m\in \N$, and therefore 
$\mu(A)=0$ for all $A\subseteq \N$.
\end{lemma}

\begin{proof}
First note that \eqref{eq:union-prime} implies 
$\mu(A_{p,k})=\mu(\N)/p$ for every prime number $p$ and all 
$k\in\{ 1,\dots,p\}$: indeed, if $p$ is a prime number, then
\[
\mu(\N) = \mu\bigg( \bigcup_{k\in\{1,\dots,p\}} A_{p,k}\bigg) = \sum_{k\in\{1,\dots,p\}} \mu(A_{p,k}) \stackrel{\eqref{eq:union-prime}}{=} p \mu(A_{p,p})\,,
\]
where we used the finite additivity of $\mu$ to obtain the second equality.
Combining $\mu(\N)=p \mu(A_{p,p})$  with 
\eqref{eq:union-prime} gives $\mu(A_{p,k})=\mu(\N)/p$ for all $k\in\{1,\dots,p \}$.

Now fix
 $m\in \N$. For every prime number  $p$ there exist numbers 
$j\in \N\cup\{0\}$ and $k\in \{1,\dots,p\}$ such that $m=jp+k$. 
Thus $m\in A_{p,k}$. From our earlier considerations it follows
\[
\mu(\{m\})\le \mu(A_{p,k})=\mu(\N)/p\,.
\]
Since $\mu(\N)<\infty$ by assumption, and since there are arbitrarily large primes, it follows that 
$\mu(\{m\})=0$. But since  $\mu$ is a measure, and thus is  $\sigma$-additive, 
we get $\mu(A)=\sum_{m\in A}\mu(\{m\})=0$ for every $A\subseteq\N$.
\end{proof}

So there exists no measure on $\Pc(\N)$ having the desired property
\eqref{eq:union-prime}. But could it be that we have chosen
the domain of $\mu$ too large? The next proposition shows that there 
is no smaller domain containing all $A_{p,k}$.
\begin{proposition}
We have
$\sigma\big(\big\{A_{p,k}\colon p\;\text{prime},\, k\in\{1,\dots,p\}\big\}\big)=\Pc(\N)$.
\end{proposition}

\begin{proof}
We define the set $\Sigma:=\sigma\big(\big\{A_{p,k}\colon p\;\text{prime},\, k\in\{1,\dots,p\}\big\}\big)$.
It is sufficient to show that $\{m\}\in\Sigma$ for all $m\in\N$. To this end fix $m\in\N$. For every prime $p>m$ we have $m\in A_{p,m}$, since $m=0\cdot p+m$. 
Therefore,
\[m\in \bigcap_{p\;\text{prime},\,p>m} A_{p,m}\,.\]
Let $\ell\in \bigcap_{p\;\text{prime},\,p>m} A_{p,m}$. Then there exists a
prime $p$ with $p>\ell$ so that, since  $\ell\in A_{p,m}$,
$\ell=0\cdot p+m=m$. 
Thus, $\{m\}= \bigcap_{p\;\text{prime},\,p>m} A_{p,m}\in \Sigma$.
\end{proof}

\begin{rem}
Equation \eqref{eq:union-prime} in Lemma \ref{lem:nomeasure} formalizes
our earlier intuition that if $\mu(A_{p,p})$ is the fraction of 
numbers divisible by $p$ then this should equal the fraction 
of numbers giving remainder 1 and so on. The  Lemma shows us that there cannot
be a non-trivial finite measure $\mu$ with this property and therefore
we cannot assign meaningful probabilities to those subsets in the 
framework of  Kolmogorov's theory.
In contrast to the independence of distinct binary digits of a number in 
$[0,1)$, 
we cannot 
cover the independence of divisibility by distinct primes of a number in $\N$
using Kolmogorov's notion of independence of random variables. 
\end{rem}

\section{Relative Measures}

A possible remedy is a notion related to one of the earlier approaches to probability theory going back at least to Richard von Mises (19. April 1883 in Lviv; 14. Juli 1953 in
Boston) and can be found in early work of Kac and Steinhaus. However, we were unable to trace the original source. In any case, this approach has to a large extend been replaced by Kolmogorov's axiomatization of probability.

One of the central notions in this manuscript shall be referred to as relative measure and its definition and properties be discussed in the following section.

\subsection{Relative measurable subsets of $\N$}


\begin{df}[Relative measurable subsets of $\N$ and relative measure]\label{def:diskr rel mass}
We say that a subset $A\subseteq \N$ is {\em relatively measurable} if and only if the limit
\[
\lim_{N\to\infty} \frac{|A\cap\{1,\dots,N\}|}{N}\,,
\]
exists. In that case we define the {\em relative measure} $\mu_R$ of $A$ as exactly this limit,
\[
\mu_R(A):=\lim_{N\to\infty} \frac{|A\cap\{1,\dots,N\}|}{N}\,.
\] 
\end{df}

It is easy to see that the collection of relatively measurable subsets of $\N$ forms an algebra and that $\mu_R$ is a non-negative and (finitely-)additive set function on it.
Moreover, it is obvious that every finite subset of $\N$ is relatively measurable with relative measure $0$. 

The sets $A_{n,k}$, $n\in\N$ and $k\in\{0,\dots,n-1\}$ defined in Subsection \ref{subsec:primfaktoren-unabh} are relatively measurable with
\[
\mu_R(A_{n,k})=\frac{1}{n}\,.
\]
It is a direct consequence of Lemma \ref{lem:nomeasure} that $\mu_R$ cannot be $\sigma$-additive. Indeed,
\[
\mu_R\Big(\bigcup_{i\in\N}\{i\}\Big) = \mu_R(\N) = 1 \neq 0 = \sum_{i\in\N} \mu_R(\{i\})\,.
\]
On the other hand, we can construct sets which are {\em not} relatively measurable.

\begin{example}
Let $a_1=0$ and define 
\[
a_{k}:=\begin{cases}
0 & : 2^{2m}< k\le 2^{2m+1}\text{ for some }m\in\N_0 \\
1 & : 2^{2m+1} <k\le 2^{2m+2}\text{ for some }m\in\N_0 \,.
\end{cases}
\]
Consider the level set $A:=\{k\in\N\,\colon\, a_k=1\}$. Then $A$ is not relatively measurable because
\begin{align*}
2^{-(2m+2)}|A\cap \{1,\dots,2^{2m+2}\}|
&=2^{-(2m+2)}2(1+2^2+\dots+2^{2m+1})
=2^{-(2m+1)}\frac{2^{2m+2}-1}{3}\to \frac{2}{3}\\
2^{-(2m+1)}|A\cap \{1,\dots,2^{2m+1}\}|
&=2^{-(2m+1)}2(1+2^2+\dots+2^{2m-1})=2^{-(2m)}\frac{2^{2m}-1}{3}\to \frac{1}{3}\,.
\end{align*}
\end{example}

\medskip
The relative measure allows us to conceive and show the independence of divisibility by different primes in a formal way. In this regard this notion is superior to a measure in the sense of Kolmogorov. We are now going to prove the independence of $A_{p,p}$ and $A_{q,q}$ for different primes $p$ and $q$. By the fundamental theorem of arithmetic a number is divisible by $p$ as well as $q$ if and only if it is divisible by their product $pq$, and so 
$A_{p,p}\cap A_{q,q}=A_{pq,pq}$.
Therefore, we obtain
\[
\mu_R(A_{p,p}\cap A_{q,q})=\mu_R(A_{pq,pq})=\frac{1}{p\cdot q}
=\frac{1}{p}\cdot \frac{1}{q}=\mu_R(A_{p,p})\,\mu_R(A_{q,q})\,,
\]
which is the product rule so characteristic for independence. Similarly, one can show this property for each finite collection of different primes $p_1,\dots,p_m$. 

The following lemma shows that if the indicator function of a subset of the natural numbers is eventually periodic, then the relative measure of that set is equal to the average over the period. We shall leave the proof to the reader.
	
\begin{lemma}\label{lemma:periodisch}
Consider a set $A\subseteq \N$. If there exist $k\in \N$ and $n_0\in \N$ such that 
\[
\forall n\ge n_0\colon 1_A(n+k)=1_A(n)\,,
\] 
then $A$ is relatively measurable and 
\[
\mu_R(A)=\frac{|A\cap \{n_0+1,\dots,n_0+k\}|}{k}.
\]
\end{lemma}

\begin{rem}[Independence and information]\label{rem:independence and information}
One important property of statistical independence is that knowledge
of one event, say $B$, does not present any information about an 
independent event $A$:
for independent $A,B$ we have $\Pro(A|B)=\Pro(A)$.

A similar situation occurs with numbers: knowledge about divisibility 
by one prime does not tell us anything about divisibility by another one.
This holds also true for the digits considered earlier:
if we know the $k$-th digit of a number $x\in [0,1)$ 
this does not tell us anything about 
its $\ell$-th digit.
\end{rem}

Consider now, for every 
 $j\in \N$ 
the function  $\beta_j\colon \N\to\{0,1\}$ defined by 
\begin{equation}
\beta_j(n):=\begin{cases}
0 & : \lfloor \frac{n}{2^{j-1}}\rfloor \text{ is even }\\
1 & : \lfloor \frac{n}{2^{j-1}}\rfloor \text{ is uneven }\,,
\end{cases}
\end{equation}
such that $\beta_j(n)$ is the $j$-th binary digit of $n$,
and
\[
n=\sum_{j=1}^\infty \beta_j(n)\,2^{j-1}
=\sum_{j=1}^{\lfloor\log_2(n)\rfloor+1} \beta_j(n)\,2^{j-1}\,.
\]
To every  $j\in \N$ assign the set $B_j:=\{n\in \N\colon
\beta_j(n)=1\}$, i.e.~the set of all natural numbers for which 
the $j$-th binary digit equals $1$.

It follows from the definition of binary digits that for each $j\in \N$
\[
B_j=\bigcup_{m\in\N\cup\{0\}}\{2^{j-1}(2m+1),\dots,2^{j-1}(2m+1)+2^{j-1}-1\}\,,
\]
which means that
\[
B_j^c=\bigcup_{m\in\N\cup\{0\}}\{2^{j}m,\dots,2^{j}m+2^{j-1}-1\}\,,
\]
and so $\mu_R(B_j)=\frac{1}{2}$. Moreover, for every choice of $j,k\in \N$ with
$j<k$, we have
$\mu_R(B_j\cap B_k)=\mu_R(B_j)\mu_R(B_k)$, which can be proven using Lemma \ref{lemma:periodisch}.

\begin{df}
Let $(A_j)_{j\in J}$ be a family of relatively measurable subsets of $\N$. We say that $(A_j)_{j\in J}$ are {\em independent} if and only if for every $m\in\N$ and every subset $I$ of cardinality $m$
\[
\mu_R\Big(\bigcap_{i\in I}A_{i}\Big)=
\prod_{i\in I}\mu_R(A_{i})\,.
\]
\end{df}

Summarizing the preceding thoughts, we obtain the following result.

\begin{proposition}
(1) For $n\in \N$ and $k\in\{1,\dots,n\}$, let
$A_{n,k}:=\{jn+k\colon j\in\N\cup\{0\} \}$. 
Then the family $\big(A_{p,p}\big)_{p\in\N,p\text{ prime}}$
is independent.
\vskip 1mm
\noindent (2) For every $j\in\N$ let
\(
B_j=\bigcup_{m\in\N\cup\{0\}}\{2^{j-1}(2m+1),\dots,2^{j-1}(2m+1)+2^{j-1}-1\}\,.
\)
Then the family $\big(B_{j}\big)_{j\in\N}$
is independent.
\end{proposition}

It is quite interesting that similar results to the ones for expansions of real numbers in $[0,1)$ with respect to the Lebesgue measure can be obtained for the expansion of natural numbers with respect to the relative measure on $\N$.

\subsection{Relatively measurable sequences and their distribution}

 
In this subsection we shall introduce the notion of a relatively measurable
sequence and, in broad similarity to the way independence is defined in the
sense of Kolmogorov, we introduce the notion of relatively independent
sequences $x,y\colon \N \to \R$ and define a distribution function with respect
to relative measures. As we shall see, such a distribution function does not
possess all the properties that --- coming from probability theory --- we might expect it to have.

\begin{df}[Relatively measurable sequence]
A sequence $x\colon \N\to \R$ is said to be {\em relatively measurable} if and only if the pre-image 
\[
x^{-1}(I):=\big\{n\in \N\colon x_n\in I\big\}
\]
of each interval $I\subseteq \R$ under $x$ is a relatively measurable subset of $\N$.
\end{df}

To us an interval means a convex subset of $\R$, in particular singleton sets are intervals. Natural examples of measurable sequences are indicator functions of relatively measurable sets and their finite sums. 

We shall now introduce what it means for two sequences to be independent with respect to a relative measure. This is again done via a product rule.

\begin{df}[Independent sequences]
Two relatively measurable sequences $x,y\colon \N\to\R$ are said to be $\mu_R$-independent if and only if for any two intervals 
$I,J\subseteq \R$ we have
\[
\mu_R\big(x^{-1}(I)\cap y^{-1}(J)\big)=\mu_R\big(x^{-1}(I)\big)\,\mu_R\big(y^{-1}(J)\big)\,.
\]
\end{df}

This definition can be generalized in an obvious way to any finite number of relatively measurable sequences.  

We now turn to the definition of a (relative) distribution function of a relatively measurable sequence. 

\begin{df}[Distribution function]
Let $x\colon \N\to\R$ be a relatively measurable sequence. Then the function 
\[F_x:\R\to [0,1],\quad
F_x(z) := \mu_R\Big(\big\{n\in\N\,:\, x_n \in(-\infty,z] \big\}\Big)
\]
is called the {\em (relative) distribution function} of $x$.
\end{df}

By its very definition such a distribution function resembles a classical
distribution function we know from probability theory. In particular, it is
immediately clear that it is non-decreasing. However, in general not all
properties we may expect from a relative distribution function have to hold.

\begin{example}
Consider the sequence $x:\N\to\R$ given by
\[
x_n:=\begin{cases}
-n & \text{ if } n=4k \text { for  some }k\in \N_0\\
0 & \text{ if } n=4 k+1  \text { for  some }k\in \N_0\\
\frac{1}{n} & \text{ if } n=4k+ 2 \text { for  some }k\in \N_0\\
n & \text{ if } n=4k+ 3 \text { for  some }k\in \N_0\,.
\end{cases}
\] 
Then it is easy to see that $x$ is relatively measurable and that its relative distribution function  is given by
\[
F_x(z)=\frac{1}{4}1_{(-\infty,0)}(z)+\frac{2}{4}1_{\{0\}}(z)+\frac{3}{4}1_{(0,\infty)}(z)\,.
\]
Hence, 
$F_x$ is neither left nor right continuous, and we have
\[ 
\lim_{z\to -\infty}F_x(z)>0 \qquad\text{and}\qquad \lim_{z\to -\infty}F_x(z)<1\,.
\]
\end{example}
Note however that for every bounded relatively measurable sequence $x$
\[ 
\lim_{z\to -\infty}F_x(z)=0 \qquad\text{and}\qquad \lim_{z\to -\infty}F_x(z)=1\,.
\]

Next we introduce and study the notion of an average of a relatively measurable sequence. 

\begin{df}[Relative average]
Let $x:\N\to\R$ be a relatively measurable sequence. Then we define the {\em relative average} of $x$ by
\[
M(x) := \lim_{N\to\infty} \frac{1}{N}\sum_{n=1}^N x_n\,,
\]
whenever this limit exists.
\end{df}

The following theorem shows that the relative average of a relatively
measurable and bounded sequence can be written in terms of a Stieltjes integral
with respect to  the relative distribution function.

\begin{thm}
Let $x:\N\to\R$ be a relatively measurable and bounded sequence. Then
$M(x)$ exists and 
\begin{equation}\label{eq:stieltjes}
M(x) = \int_{-\infty}^\infty z \,\dint F_x(z)\,.
\end{equation}
\end{thm}
\begin{proof}
In this proof we simply write $F$ instead of $F_x$. By assumption there exists some $K\in (0,\infty)$ such that $-K+1\le x_n\le K$
for every $n\in\N$.  The Stieltjes integral exists since the function
$\id\colon [-K,K]\to \R$, $z\mapsto z$ is continuous and $F$ is monotone 
on $[-K,K]$ and constant on the
intervals $(-\infty,-K]$ and $[K,\infty)$.
Therefore, given $\varepsilon>0$ there exists a decomposition
$Z=\{-K=t_0<t_1<\ldots<t_m=K\}$ of $[-K,K]$ such that
$O(\id,F,Z)-U(\id,F,Z)<\varepsilon$, where $U$ and $O$ denote upper and lower
Riemann-Stieltjes sums, i.e., 
\begin{align*}
U(\id,F,Z)&=\sum_{k=1}^m t_{k-1}\big(F(t_k)-F(t_{k-1})\big)
\quad\text{ and }
&O(\id,F,Z) =\sum_{k=1}^m t_k\big(F(t_k)-F(t_{k-1})\big)\,.
\end{align*}
We observe that
\begin{align*}
\limsup _{N\to\infty} \frac{1}{N}\sum_{n=1}^N x_n
&= \sum_{k=1}^m \limsup_{N\to\infty} \frac{1}{N}\sum_{n=1}^N1_{(t_{k-1},t_k]}(x_n)x_n
\le  \sum_{k=1}^m \limsup_{N\to\infty} \frac{1}{N}\sum_{n=1}^N1_{(t_{k-1},t_k]}(x_n)t_k\\
&\le  \sum_{k=1}^m t_k\big(F(t_k)-F(t_{k-1})\big)=O(\id,F,Z)\,.
\end{align*}
Similarly one can show that $\liminf_{N\to\infty}\frac{1}{N}\sum_{n=1}^N x_n\ge U(\id,F,Z)$, which then proves the assertion.
\end{proof}

It follows from the properties of Riemann-Stieltjes integrals that
\begin{equation}\label{eq:stieltjes1}
M(x) = \int_{-\infty}^\infty z F_x'(z)\,\dint z\,,
\end{equation}
whenever $F_x$ is differentiable on $\R$ with $F'_x=f_x$ outside some
at most finite subset of $\R$. 

\begin{rem}
We see that measurable sequences behave in many ways like random variables, and indeed a measurable sequence can be taken as a mathematical model for a ``random number''. As noted before, this kind of model has been put forward by Austrian mathematician Richard von Mises in the first half of the 20th century. 
This model was --- at least among the vast majority of probabilists --- replaced by Kolmogorov's approach, mainly because of the potent tools from Lebesgue's measure theory and the accompanied clean and simple concepts and theorems of convergence.

Nevertheless there is a certain appeal to the alternative, in particular its sleek theoretical foundation. Within this approach one can simply state that a {\em real} number is a Cauchy sequence of rational numbers and a {\em random} number is a relatively measurable sequence of
rational (or real) numbers. 
\end{rem}

\medskip
We now assign to every  $\Z$-valued and relatively measurable sequence $x$ 
a function
$\rho_x\colon \Z\to [0,1]$ via
\[
\rho_x(k):=\mu_R\big(\{n\in\N\colon x_n=k\}\big)\,.
\]
Then for bounded, $\Z$-valued and relatively measurable sequences we have 
$\sum_{k\in\Z} \rho_x(k)=1$ and  the well-known convolution formula:

\begin{proposition}
Let $x,y:\N\to\R$ be bounded and relatively measurable sequences taking values in $\Z$.
If $x$ and $y$ are $\mu_R$-independent, then $\rho_{x+y}=\rho_x\ast \rho_y$,
where
\[
\rho_x\ast \rho_y(k):=\sum_{j\in\Z} \rho_x(j)\rho_y(k-j)\,,\quad k\in\Z\,.
\]
\end{proposition}

All in all, we can say that relatively measurable sequences behave in many ways like random variables. For instance, the indicator functions of the sets $B_j$ introduced after Remark \ref{rem:independence and information} form an independent, relatively measurable, bounded, and $\Z$-valued sequence. Therefore, their sums satisfy
\[
\rho_{1_{B_1}+\ldots+1_{B_m}}(k)={m \choose k}2^{-m}\,,\quad k\in\Z\,.
\]
This means that the partial sums of the indicator functions of the sets
$B_j$ satisfy the central limit theorem of de Moivre (Theorem \ref{clt:de moivre}), i.e., for any $a,b\in\R$ with $a<b$,
\begin{align}\label{clt:ziffernentwicklung}
\begin{split}
\lefteqn{\lim_{m \to \infty} 
\mu_R\bigg(\bigg\{n\in \N\colon a \leq \frac{\sum_{j=1}^m1_{B_j}(n) - \frac{m}{2}}{\sqrt{\frac{m}{4}}}
\leq b\bigg\}\bigg) } \\
&=\lim_{m \to \infty}\sum_{k=0}^m{m \choose k}2^{-m}1_{[a,b]}\Big(\frac{k - \frac{m}{2}}{\sqrt{\frac{m}{4}}}\Big)
= \frac{1}{\sqrt{2\pi}}\int_a^b e^{-\frac{x^2}{2}}\,\dint x\,.
\end{split}
\end{align}
Note again that the set considered above is indeed relatively measurable. To see this, we note that, as was argued before, the sets $B_j$ are all relatively measurable and hence, because the collection of relatively measurable sets forms an algebra, so are their complements $B_j^c$. This immediately implies that $(1_{B_j}(n))_{n\in\N}$ and their finite sums are relatively measurable. 

Thus, for the binary expansion of natural numbers we have the same central limit theorem as for the binary expansion of real numbers in $[0,1)$. 
In fact, we can now formulate a quite interesting version of this, which can 
be found, for example, in \cite{DG1998}. Contrary to almost all
numbers in $[0,1)$, every natural number has a finite expansion and hence it is reasonable to define for $n\in\N$ its \sod{}  with respect to the binary expansion,
\[s_2(n):=\sum_{j=1}^{\lfloor\log_2(n)\rfloor+1} 1_{B_j}(n)=\sum_{j=1}^\infty
1_{B_j}(n) \,,\quad n\in\N\,.
\]
The following result describes the Gaussian fluctuations of the \sod{} .
\begin{thm}[Central limit theorem for the \sod{}]\label{th:clt-s2n}
For all $b\in\R$, we have
\[
\mu_R\Big(\Big\{n\in\N\colon s_2(n)\le b\sqrt{\tfrac{1}{4}\log_2(n)}+\tfrac{1}{2}\log_2(n)\Big\}\Big) = \frac{1}{\sqrt{2\pi}}\int_{-\infty}^b e^{-\frac{x^2}{2}}\,\dint x\,.
\]
\end{thm}

We recall the following lemma from probability theory.

\begin{lemma}\label{lem:uniform-conv}
Let $F:\R\to [0,1]$ be a continuous cumulative distribution function and 
let $(F_n)_{n\in \N}$ be a sequence of non-decreasing functions 
$F_n\colon \R\to [0,1]$ with 
$\lim_{n\to\infty}F_n(x)=F(x)$ for all $x\in \R$.
Then $F_n\to F$ uniformly on $\R$.
\end{lemma}


We are now able to prove the central limit theorem for the sum-of-digits function.

\begin{proof}[Proof of Theorem \ref{th:clt-s2n}]
Let $\varepsilon\in(0,\infty)$. For  $b\in\R$ let us write 
$\Phi(b):=\frac{1}{\sqrt{2\pi}}\int_{-\infty}^b e^{-\frac{x^2}{2}}\,\dint x$. 
It follows from de Moivre's central limit theorem (see Equation
\eqref{clt:ziffernentwicklung}) and Lemma \ref{lem:uniform-conv} that there exists
$m_0\in\N$ such that for all $m\geq m_0$ and every $b\in \R$, \[ 
-\frac{\varepsilon}{6}<\mu_R\bigg(\bigg\{n\in \N\colon \frac{\sum_{j=1}^m1_{B_j}(n) - \frac{m}{2}}{\sqrt{\frac{m}{4}}}
\leq b\bigg\}\bigg) -\Phi(b)<\frac{\varepsilon}{6}\,.
\]
Moreover, for each $m\ge m_0$, we have 
\[
\left|\bigg\{0\le n< 2^m\colon\sum_{j=1}^{m}s_2(n)=k\bigg\}\right|=\left|\bigg\{0\le n< 2^m\colon\sum_{j=1}^{m}1_{B_j}(n)=k\bigg\}\right|={m\choose k}
\]
and therefore, 
\[
2^{-m}\Big|\Big\{0\le n< 2^m\colon s_2(n)\le b\sqrt{\tfrac{1}{4}m}+\tfrac{1}{2}m \Big\}\Big|\in (\Phi(b)-\tfrac{\varepsilon}{6},\Phi(b)+\tfrac{\varepsilon}{6})\,.
\]
Now let $\ell\in \N$ with $2^{-\ell}<\tfrac{\varepsilon}{3}$ and $j\in \{1,\ldots,2^\ell\}$. For every $m\ge \ell+m_0$,
\begin{align*}
\lefteqn{\tfrac{1}{j 2^{m-\ell}}\Big|\Big\{2^m\le n< 2^m+j 2^{m-\ell}\colon s_2(n)\le b\sqrt{\tfrac{1}{4}\log_2(n)}+\tfrac{1}{2}\log_2(n)\Big\}\Big|}\\
&\ge \tfrac{1}{j 2^{m-\ell}}\Big|\Big\{2^m\le n< 2^m+j 2^{m-\ell}\colon s_2(n)\le b\sqrt{\tfrac{1}{4}m}+\tfrac{1}{2}m\Big\}\Big|\\
&=\sum_{i=1}^j\tfrac{2^{m-\ell}}{j 2^{m-\ell}}2^{-(m-\ell)}\Big|\Big\{0\le n< 2^{m-\ell}\colon s_2(n)\le b\sqrt{\tfrac{m}{4}}+\tfrac{m}{2}-s_2(i)\Big\}\Big|\,.
\end{align*}
Since $m-\ell\ge m_0$,
\begin{align*}
\lefteqn{2^{-(m-\ell)}\Big|\Big\{0\le n< 2^{m-\ell}\colon s_2(n)\le b\sqrt{\tfrac{m}{4}}+\tfrac{m}{2}-s_2(i)\Big\}\Big|}\\
&\ge \Phi\Big(b\sqrt{\tfrac{m}{m-\ell}}+\tfrac{\ell-2s_2(i)}{\sqrt{m-\ell}}\Big)-\tfrac{\varepsilon}{6}
\ge \Phi\Big(b\sqrt{\tfrac{m}{m-\ell}}-\tfrac{\ell}{\sqrt{m-\ell}}\Big)-\tfrac{\varepsilon}{6}
\,.
\end{align*}
Therefore,
\begin{align*}
&\tfrac{1}{j 2^{m-\ell}}\Big|\Big\{2^m\le n< 2^m+j 2^{m-\ell}\colon s_2(n)\le b\sqrt{\tfrac{1}{4}\log_2(n)}+\tfrac{1}{2}\log_2(n)\Big\}\Big|\\
&\ge \Phi\Big(b\sqrt{\tfrac{m}{m-\ell}}-\tfrac{\ell}{\sqrt{m-\ell}}\Big)-\tfrac{\varepsilon}{6} \,,
\end{align*}
and in the same way,
\begin{align*}
\lefteqn{\tfrac{1}{j 2^{m-\ell}}\Big|\Big\{2^m\le n< 2^m+j 2^{m-\ell}\colon s_2(n)\le b\sqrt{\tfrac{1}{4}\log_2(n)}+\tfrac{1}{2}\log_2(n)\Big\}\Big|}\\
&=\tfrac{1}{j 2^{m-\ell}}\Big|\Big\{2^m\le n< 2^m+j 2^{m-\ell}\colon s_2(n)\le b\sqrt{\tfrac{1}{4}(m+1)}+\tfrac{1}{2}(m+1)\Big\}\Big|\\
&\le \Phi\Big(b\sqrt{\tfrac{m+1}{m-\ell}}+\tfrac{1+\ell}{\sqrt{m-\ell}}\Big)+\tfrac{\varepsilon}{6} \,.
\end{align*}
Now for fixed $b\in \R$ there exists $m_1\in \N$ with $m_1\ge m_0+\ell$ such that for all $m\ge m_1$
\begin{align*}
\Phi(b)-\tfrac{\varepsilon}{3}<\tfrac{1}{j 2^{m-\ell}}\Big|\Big\{2^m\le n< 2^m+j 2^{m-\ell}\colon s_2(n)\le b\sqrt{\tfrac{1}{4}\log_2(n)}+\tfrac{1}{2}\log_2(n)\Big\}\Big|
&<\Phi(b)+\tfrac{\varepsilon}{3} \,.
\end{align*}
Note that this equation holds in particular for $j=2^\ell$, so that 
\begin{align*}
\Phi(b)-\tfrac{\varepsilon}{3}<\tfrac{1}{2^{m}}\Big|\Big\{2^m\le n< 2^{m+1}\colon s_2(n)\le b\sqrt{\tfrac{1}{4}\log_2(n)}+\tfrac{1}{2}\log_2(n)\Big\}\Big|
&<\Phi(b)+\tfrac{\varepsilon}{3} \,.
\end{align*}
Now let $N>2^{m_1}\tfrac{3}{\varepsilon}$, and let $m=\lfloor \log_2(N)\rfloor$.
Then $2^m+(j-1)2^{m-\ell}\le N<2^m+j2^{m-\ell}$ for some $j\in\{1,\dots,2^\ell\}$. Then,
\begin{align*}
\lefteqn{\tfrac{1}{N}\Big|\Big\{0\le n< N\colon s_2(n)\le b\sqrt{\tfrac{1}{4}\log_2(n)}+\tfrac{1}{2}\log_2(n)\Big\}\Big|}\\
&=
\tfrac{1}{N}\Big|\Big\{0\le n< 2^{m_1}\colon s_2(n)\le b\sqrt{\tfrac{1}{4}\log_2(n)}+\tfrac{1}{2}\log_2(n)\Big\}\Big|\\
&\quad+\sum_{k=m_1}^{m-1}\tfrac{2^k}{N}\tfrac{1}{2^k}\Big|\Big\{2^k\le n< 2^{k+1}\colon s_2(n)\le b\sqrt{\tfrac{1}{4}\log_2(n)}+\tfrac{1}{2}\log_2(n)\Big\}\Big|\\
&\quad+1_{\{j\ne -1\}}\tfrac{(j-1) 2^{m-\ell}}{N}\tfrac{1}{(j-1) 2^{m-\ell}}\Big|\Big\{2^m\le n< 2^m+(j-1) 2^{m-\ell}\colon s_2(n)\le b\sqrt{\tfrac{1}{4}\log_2(n)}+\tfrac{1}{2}\log_2(n)\Big\}\Big|\\
&\quad+\tfrac{1}{N}\Big|\Big\{2^m+(j-1)2^{m-\ell}\le n< N\colon s_2(n)\le b\sqrt{\tfrac{1}{4}\log_2(n)}+\tfrac{1}{2}\log_2(n)\Big\}\Big|\\
&\le \tfrac{\varepsilon}{3}
+\tfrac{1}{N}\sum_{k=0}^{m-1}2^k\big(\Phi(b)+\tfrac{\varepsilon}{3}\big)
+1_{\{j\ne -1\}}\tfrac{(j-1) 2^{m-\ell}}{N}\big(\Phi(b)+\tfrac{\varepsilon}{3}\big)+2^{m-\ell}\tfrac{1}{N}\\
&<\tfrac{\varepsilon}{3}+\tfrac{2^m+(j-1)2^{m-\ell}}{N}\big(\Phi(b)+\tfrac{\varepsilon}{3}\big)+\tfrac{\varepsilon}{3}\le \Phi(b)+\varepsilon\,,
\end{align*}
where we have used that since $2^{-\ell}<\tfrac{\varepsilon}{3}$, we also have 
$2^{m-\ell}\tfrac{1}{N}\le 2^{m-\ell}\tfrac{1}{2^m}<\tfrac{\varepsilon}{3}$.
In the same way we get
\begin{align*}
\lefteqn{\tfrac{1}{N}\Big|\Big\{0\le n< N\colon s_2(n)\le b\sqrt{\tfrac{1}{4}\log_2(n)}+\tfrac{1}{2}\log_2(n)\Big\}\Big|}\\
&\ge 
\tfrac{1}{N}\sum_{k=m_1}^{m-1}2^k\big(\Phi(b)-\tfrac{\varepsilon}{3}\big)
+1_{\{j\ne -1\}}\tfrac{(j-1) 2^{m-\ell}}{N}\big(\Phi(b)-\tfrac{\varepsilon}{3}\big)\\
&=\tfrac{2^m-2^{m_1}+(j-1)2^{m-\ell}}{N}\big(\Phi(b)-\tfrac{\varepsilon}{3}\big)
=\big(1-\tfrac{N-2^m+2^{m_1}-(j-1)2^{m-\ell}}{N})\big(\Phi(b)-\tfrac{\varepsilon}{3}\big)\\
&=\Phi(b)-\tfrac{2^{m_1}}{N}-\tfrac{\varepsilon}{3}-\tfrac{N-2^m-(j-1)2^{m-\ell}}{N}
>\Phi(b)-2\tfrac{\varepsilon}{3}-\tfrac{2^{m-\ell}}{N}
> \Phi(b)-\varepsilon\,, \end{align*}
which proves the result.
\end{proof}

\medskip
\subsection{Uniform distribution mod 1 and Weyl's theorem}

In this section we address a famous theorem of Hermann Weyl (9. November 1885 in Elmshorn; 8. December 1955 in Zürich). Before we start, let us remind the reader that the fractional part of a number $x\in\R$ is defined as 
\[
\{x \} :=
x - \lfloor x \rfloor 
\]
where
\[
\lfloor x \rfloor:= \max\{k\in\mathbb Z\,:\, k\leq x \} 
\,.
\]
If we are given a sequence $x:\N\to\R$ and a set $B\subseteq [0,1)$, then we define another set by setting
\[
A_{x,B}:= \big\{n\in\N \,:\, \{x_n\}\in B \big\}\,.
\]
The sequence $x=(x_n)_{n\in\N}$ is said to be uniformly distributed modulo $1$ (we simply write mod 1) if and only if for all $a,b\in\R$ with $0\leq a < b\leq 1$, we have
\[
\mu_R\big(A_{x,[a,b)}\big) = b-a\,.
\]
In particular, this means that for each uniformly distributed sequence 
$(x_n)_{n\in\N}$ the sequence $(\{x_n\})_{n\in\N}$ is relatively measurable. 

Weyl's theorem  \cite{W1914, W1916}, also known as Weyl's criterion, says that a sequence $(x_n)_{n\in\N}$ of real numbers is uniformly distributed mod 1 if and only if for every $h\in\mathbb Z\setminus\{0\}$ the following condition is satisfied,
\[
\lim_{N\to\infty} \frac{1}{N}\sum_{n=1}^N e^{2\pi i h x_n} = 0\,.
\]
In an extended and multivariate version this theorem reads as follows.

\begin{thm}\label{th:weyl}
Let $m\in\N$ and consider sequences $x^1,\ldots,x^m:\N \to\R$ . Then the following are equivalent:
\begin{enumerate}
\item Every sequence $x^k$, $k\in\{ 1,\ldots,m\}$ is uniformly distributed mod 1 and
$\{x^1\},\ldots,\{x^m\}$ are $\mu_R$-independent;
\item For each $m$-tuple $(h_1,\ldots,h_m)\in \Z^m\setminus\{0\}$,
\[
\lim_{N\to\infty} \frac{1}{N}\sum_{n=1}^N e^{2\pi i (h_1 x^1_n+\ldots+h_m x^m_n)} = 0\,;
\]
\item For every continuous function $\psi\colon [0,1]^m\to\R$,
\[
\lim_{N\to\infty} \frac{1}{N}\sum_{n=1}^N \psi(\{x^1_n\},\ldots,\{x^m_n\})
=\int_{[0,1]^m} \psi(z_1,\ldots,z_m)\,\dint z_1\ldots\dint z_m\,;
\]
\item For every Riemann integrable function $\psi\colon [0,1]^m\to\R$, 
\[
\lim_{N\to\infty} \frac{1}{N}\sum_{n=1}^N \psi(\{x^1_n\},\ldots,\{x^m_n\})
=\int_{[0,1]^m} \psi(z_1,\ldots,z_m)\,\dint z_1\ldots\dint z_m\,.
\]\end{enumerate} 
\end{thm}

An important consequence is that for each 
$\alpha\in \R$ the sequence $(n\alpha)_{n\in \N}$ is uniformly distributed mod 1 if and only if $\alpha$ is irrational, and that for 
$\alpha_1,\ldots,\alpha_m\in \R$
the sequences $\{\alpha_1n\}_{n\in\N},\ldots,\{\alpha_mn\}_{n\in\N}$
are uniformly distributed mod 1 and $\mu_R$-independent if and only if
$1,\alpha_1,\ldots,\alpha_m$ are linearly independent over $\Q$.

\begin{rem}
Theorem \ref{th:weyl} is also of practical interest, as it provides us 
with a method for numerical integration of a 
Riemann integrable function $\psi$ on $[0,1]^m$. Note that,
if we only know that the coordinate sequences are uniformly distributed
mod 1 and $\mu_R$-independent, we cannot say anything about the speed
of convergence of the sums towards the integral. 

The concept of discrepany of a sequence measures the speed with which
a sequence in $[0,1)^m$ approaches the uniform distribution on
$[0,1)^m$. Sequences with a ``high'' speed of convergence are informally
called low-discrepancy sequences and give rise to a class of 
numerical integration algorithms called quasi-Monte Carlo methods.
For more information about these sequences and algorithms see
\cite{dick, DT1997, kuipersnied,leobacher}.
%
\end{rem}
\begin{df}[Finitely measurable function]
We say that a function $g\colon I\to \R$ is {\em \frm{} } if and only if the pre-image of each interval $J\subset\R$ under $g$ can be written as the union of finitely many subintervals, i.e., there exists $k\in\N$ and subintervals $I_1,\ldots,I_k$ of $I$ such that
\[
g^{-1}(J)=I_1\cup\ldots\cup I_k\,.
\]
\end{df}

Examples of \frm{}  functions are the monotone functions and the functions $g$ with the following so-called Dirichlet property:
\vskip 1mm 
A function $g\colon[a,b]\to \R$ is said to have the  Dirichlet property 
if and only if  it is continuous on $[a,b]$ and has only finitely many local extreme points.
\vskip 1mm
 
A concrete example of a \frm{}  function thus is 
$\cos(2\pi\cdot)\colon [0,1]\to \R$, $z\mapsto \cos(2\pi z)$.

\begin{proposition}\label{prop:map-indep}
Let $m\in\N$ and $x^1,\ldots,x^m:\N\to\R$ be sequences. Consider \frm{}  functions $g^1,\ldots,g^m\colon \R\to \R$. 
If $x^1,\ldots,x^m$ are relatively measurable and $\mu_R$-independent, then the sequences $g^1(x^1),\ldots,g^m(x^m)$ are relatively measurable and $\mu_R$-independent.
\end{proposition}

The previous result, whose proof is left to the reader,
 has the following interesting corollary.

\begin{cor}
Let $1,\alpha_1,\dots,\alpha_m\in\R$ be linearly independent over $\Q$.
Then the sequences $\big(\cos(2\pi \alpha_1 n)\big)_{n\in\N},\ldots,\big(\cos(2\pi \alpha_m n)\big)_{n\in\N}$ are relatively measurable and $\mu_R$-independent.
\end{cor}

\begin{proof}
We have already concluded, as a consequence of Weyl's theorem, that the sequences 
$\{\alpha_1n\}_{n\in\N},\ldots,\{\alpha_mn\}_{n\in\N}$ are uniformly distributed mod 1 and $\mu_R$-independent. Hence, by
Proposition \ref{prop:map-indep} the sequences 
 \[\big(\cos(2\pi\{\alpha_1n\})\big)_{n\in\N},\ldots,\big(\cos(2\pi\{\alpha_mn\})\big)_{n\in\N}\] are $\mu_R$-independent as well and thus the sequences
\[\big(\cos(2\pi\alpha_1n)\big)_{n\in\N},\ldots,\big(\cos(2\pi\alpha_mn)\big)_{n\in\N}\,.\]
\end{proof}

\begin{proposition}
Let $x,y:\N\to\R$ be bounded and relatively measurable sequences with continuousand increasing distribution functions
$F_x$ and $F_y$ respectively. If $x$ and $y$ are $\mu_R$-independent, then the distribution function
$F_{x+y}$ 
of $x+y$ is given by the convolution of $F_x$ and $F_y$, i.e., 
\[
F_{x+y}(z)=F_x*F_y(z)
=\int_{-\infty}^\infty F_x(z-\eta)\dint F_y(\eta)
=\int_{-\infty}^\infty F_y(z-\xi)\dint F_x(\xi)\,.
\] 
\end{proposition}

\begin{proof}
It is comparably easy to see that the sequences $(F_x(x_n))_{n\in \N}$ and 
$(F_y(y_n))_{n\in \N}$ are uniformly distributed mod 1. 
Proposition \ref{prop:map-indep} implies that they are $\mu_R$-independent. 
Observe that the restriction of $F_x$ to 
the closure of $\{t\in\R\colon F_x(t)\in(0,1)\}$ is continuous and increasing
and therefore has an inverse, which we denote by $G_x$.  
Denote by $G_y$  the corresponding inverse function of $F_y$. We have
\begin{align*}
\mu_R(x+y\le z)
&=\lim_{N\to \infty}\sum_{n=1}^N1_{(-\infty,z]}(x_n+y_n)
=\lim_{N\to \infty}\sum_{n=1}^N1_{(-\infty,z]}\Big(G_x\big(F_x(x_n)\big)+G_y\big(F_y(y_n)\big)\Big)\\
&\stackrel{(*)}{=}\int_{[0,1]^2} 1_{(-\infty,z]}\big(G_x(\xi)+G_y(\eta)\big)\dint \xi\,\dint \eta
=\int_{\R^2} 1_{(-\infty,z]}(\xi+\eta)\dint F_x(\xi)\dint F_y(\eta)\\
&=\int_{-\infty}^\infty\int_{-\infty}^{z-\eta} \dint F_x(\xi)\dint F_y(\eta)
=\int_{-\infty}^\infty F_x(z-\eta)\dint F_y(\eta)\,,
\end{align*}
where we have used in $(*)$ that $(F_x(x_n))_{n\in \N}$ and 
$(F_y(y_n))_{n\in \N}$ are uniformly distributed mod 1 and independent. 
\end{proof}
If we consider, for instance, the sequence $x=\big(\cos(2\pi \alpha n)\big)_{n\in\N}$ with
irrational $\alpha$, then, since $(\alpha n)_{n\in\N}$ is uniformly distributed mod 1, 
\begin{align*}
F_x(z)&=\mu_R(x\le z)
=\lim_{N->\infty}\frac{1}{N}\sum_{n=1}^N 1_{(-\infty,z]}\big(\cos(2\pi \alpha n)\big)
=\int_0^11_{(-\infty,z]}\big(\cos(2\pi \xi)\big)\dint \xi\\
&=2\int_0^\frac{1}{2}1_{(-\infty,z]}\big(\cos(2\pi \xi)\big)\dint \xi
=\frac{1}{\pi}\int_{1}^{-1}1_{(-\infty,z]}(\eta)\arccos'(\eta)\dint \eta\\
&=\frac{1}{\pi}\int_{-1}^{1}1_{(-\infty,z]}(\eta)\arcsin'(\eta)\dint \eta
=1_{[-1,1]}(z)\frac{1}{\pi}\arcsin(z)+1_{(1,\infty)}(z)\,.
\end{align*}
This means that the distribution function of the sequence 
$\Big(\cos(2\pi \alpha_1 n)+\ldots+\cos(2\pi \alpha_m n)\Big)_{n\in\N}$
is given by $F_x^{*m}$. Therefore, we obtain a central limit theorem for partial sums of cosines with linearly independent frequencies, i.e., 
with $1,\alpha_1,\alpha_2,\dots$ linearly independent over $\Q$,
\[
\lim_{m\to\infty}\mu_R\bigg(\bigg\{n\in\N\colon a\le \frac{\cos(2\pi \alpha_1 n)+\ldots+\cos(2\pi \alpha_m n)}{\sqrt{m/2}}\le b \bigg\}\bigg)=\frac{1}{\sqrt{2\pi}}\int_a^b e^{-\frac{\xi^2}{2}}\dint\xi\,.
\]

\subsection{Relatively measurable subsets of $(0,\infty)$ --  the continuous setting}\label{subsec: continuous setting}
The deliberations of the previous subsection can quite effortlessly be lifted to a continuous setting. A continuous version of a relative measure on Lebesgue measurable subsets of $\R$ can be defined as the limit 
\[
\mu_R(A):= \lim_{T\to\infty} \frac{1}{T}\int_{0}^T \ind_{A}(x)\,\dint x
\]
if it exists. 
In analogy to the case of sequences, one obtains a continuous version of Weyl's
theorem (see also \cite[Chapter 9]{KN1974}) and thus the independence of
functions of uniformly distributed functions. An example is again given by the
cosines with linearly independent frequencies (cf. \cite{KacSteinhaus1938}),
i.e., if $1,\alpha_1,\alpha_2,\dots$ are linearly independent over $\Q$, then for
all $m\in\N$ and all $s_1,\dots,s_m\in\R$, \begin{align*}
& \mu_R\Big( \Big\{ t\in(0,\infty)\,:\,\cos(2\pi \alpha_1 t) \leq s_1, \cdots, \cos(2\pi \alpha_m t) \leq s_m \Big\}\Big) \cr
&\qquad\qquad= \prod_{j=1}^m\mu_R\Big(\big\{t\in(0,\infty)\,:\,\cos(2\pi \alpha_j t) \leq s_j \big\}\Big).
\end{align*}
Those considerations then yield a central limit theorem of the form
\[
\lim_{m\to\infty}\mu_R\left(\left\{ t\in(0,\infty)\,:\, a \leq \frac{\cos(2\pi \alpha_1 t) + \dots + \cos(2\pi \alpha_m t)}{\sqrt{m/2}} \leq b \right\} \right) = \frac{1}{\sqrt{2\pi}}\int_a^b e^{-\frac{\xi^2}{2}}\,\dint\xi.
\]
The original approach to this result is, as we find, more complicated and can be found in \cite{KacSteinhaus1938}. The latter is presented in a more accessible way in \cite[Chapter 3]{MK1959}.

\section{The Erd\H{o}s-Kac Theorem}

This section is devoted to a famous theorem of Paul Erd\H{o}s and Mark Kac. One
can say that this result marks the birth of what is today known as
probabilistic number theory. The close link between probability theory and
number theory illustrated by this theorem can hardly be overrated and turned
out to be extremely fruitful.

We shall start with the original heuristics of Mark Kac, which led him to conjecture the result he later proved together with Paul Erd\H{o}s. 

\subsection{Heuristics -- Independence \& CLT}

A guiding idea of Mark Kac has been that if there is some sort of independence, then there is the Gaussian law of errors at play. Exactly this maxim underlies the the Erd\H{o}s-Kac theorem. The object of interest is the number of different prime factors of a given number.

Let us consider the following indicator functions. For each prime number $p$ and every $n\in\N$, we define
\[
I_p(n) = 
\begin{cases}
1 &: \text{if $p$ divides $n$}\\
0 & : \text{if $p$ does not divide $n$}.
\end{cases}
\]
Given a natural number $n\in\N$, we denote by $\omega(n)$ the number of different prime factors of $n$. The indicator functions allow us to express $\omega(n)$ as follows,
\[
\omega(n) = \sum_{p \text{ prime}} I_p(n)\,.
\]
From Subsection \ref{subsec:primfaktoren-unabh} we already know that this collection of indicator functions is $\mu_R$-independent. We now want to provide a plausibility argument, and here we follow Mark Kac's original heuristics, that suggests these indicator functions also satisfy Lindeberg's condition. In analogy to the central limit theorem of Lindeberg, this suggests that the properly normalized sum of indicator functions follows a Gaussian law of errors.
For this we note first that for all  $x\in\R$ with $x\geq 2$ we have
\begin{align}\label{eq: sum recipr. primes}
\sum_{p\text{ prime,}\atop{p \leq x}} \frac{1}{p} >  \ln \ln x  -\frac{1}{2}\,,
\end{align}
see \cite[Kapitel 3]{I1992}. As we already explained in the first part of
Subsection  \ref{subsec:primfaktoren-unabh}, essentially a fraction of $1/p$ of
the numbers is divisible by the prime $p$, i.e., we may say that a number
$n\in\N$ is divisible by $p$ with probability $1/p$. In other words, the
indicator functions $I_p(n)$ behave like Bernoulli random variables
with parameter $1/p$ and are independent. But then the expectation is $1/p$ and
the variance $1/p(1-1/p)$. What does it mean for Lindeberg's condition? Well, using
the notation
of Theorem \ref{thm:lindeberg}, we have for all $n\geq 2$ \begin{align*}
s_n & = \sqrt{\sum_{p\text{ prime}\atop p\leq n} \Var[I_p(n)] } = \sqrt{\sum_{p\text{ prime}\atop p\leq n} \frac{1}{p}\bigg(1-\frac{1}{p}\bigg)} \cr 
& \geq \frac{1}{\sqrt{2}} \sqrt{\sum_{p\text{ prime}\atop p\leq n} \frac{1}{p}} \stackrel{\eqref{eq: sum recipr. primes}}{\geq} \frac{1}{\sqrt{2}} \sqrt{\ln \ln n  -\frac{1}{2}}\,.
\end{align*}
So if $\varepsilon\in(0,\infty)$, then for sufficiently large $n\in\N$, we have
\[
\E\Big[ I_p(n)^2\,\ind_{\{|I_p(n)|>\varepsilon s_n \}}\Big] \leq \Pro\big[ I_p(n)>\varepsilon s_n \big] \leq \Pro\Big[ I_p(n) >  \frac{\varepsilon}{\sqrt{2}} \sqrt{\ln \ln n  -1/2} \,\Big] = 0.
\]
The latter holds since $I_p(n)$ only takes the values $0$ and $1$. Therefore, Lindeberg's condition in Theorem \ref{thm:lindeberg} is satisfied. Together with the independence of the indication functions $I_p(n)$, $p$ prime as well as property \eqref{eq: sum recipr. primes}, this suggests that the sequence
\[
\frac{\omega(n) - \ln\ln n}{\sqrt{\ln\ln n}},\qquad n\in\N
\]
satisfies a central limit theorem. Indeed, for every $m\in\N$ let 
$c_m=\sum_{p\; \text{prime},\,p\le m}\frac{1}{p}$
and $d_m^2=\sum_{p\; \text{prime},\,p\le m}\frac{1}{p}\big(1-\frac{1}{p}\big)$.
Further let, $\omega_m(n):=\sum_{p \text{ prime},\, p\le m} I_p(n)$.
Then for every $a,b\in\R$ with $a<b$,
\[
\lim_{m\to\infty} 
\mu_R\bigg(\Big\{n\in\N\colon a\le \frac{\omega_m(n)-c_m}{d_m}\le b\Big\}\bigg)= \frac{1}{\sqrt{2\pi}}\int_a^b e^{-x^2/2}\,\dint x\,,
\]
which appears as Lemma 1 in \cite{EK1940}. This means that
\[
\lim_{m\to\infty}\lim_{N\to\infty} \frac{1}{N}\bigg|\Big\{n\in\{1,\dots,N\}\colon a\le \frac{\omega_m(n)-c_m}{d_m}\le b\Big\}\bigg| = \frac{1}{\sqrt{2\pi}}\int_a^b e^{-x^2/2}\,\dint x\,.
\]
If one could show that the two limits may be taken simultaneously, then we would obtain
\[
\lim_{N\to\infty} \frac{1}{N}\bigg|\Big\{n\in\{1,\dots,N\}\colon a\le \frac{\omega_N(n)-c_N}{d_N}\le b\Big\}\bigg| = \frac{1}{\sqrt{2\pi}}\int_a^b e^{-x^2/2}\,\dint x\,.
\]
Together with the (proper) asymptotics for $\omega_N(n), c_N, d_N$, this would give
\[
\lim_{N\to\infty} \frac{1}{N}\bigg|\Big\{n\in\{1,\dots,N\}\colon a\le \frac{\omega(n)-\ln\ln N}{\sqrt{\ln\ln N}}\le b\Big\}\bigg| = \frac{1}{\sqrt{2\pi}}\int_a^b e^{-x^2/2}\,\dint x\,.
\]
Of course, this is merely a heuristic argument, not a proof. In any case, the heuristic and conjecture just presented leads us in the following subsection to the ingenious and famous central limit theorem of Erd{\H o}s-Kac \cite{EK1940}.

\subsection{The CLT of Erd{\H o}s-Kac}

After having presented the heuristic of Mark Kac, let us tell the anecdote about the origin of the Erd{\H o}s-Kac theorem as described by Mark Kac himself in his autobiography \cite{MK1985}.
\vskip 2mm
``\emph{I knew very little number theory at the time, and I tried to find a proof along purely probabilistic lines but to no avail. In March 1939 I journeyed from Baltimore to Princeton to give a talk. Erd{\H o}s, who was spending the year at the Institute for Advanced Study, was in the audience but he half-dozed through most of my lecture; the subject matter was too far removed from his interests. Toward the end I describes briefly my difficulties with the number of prime divisors. At the mention of number theory Erd{\H o}s perked up and asked me to explain once again what the difficulty was. Within the next few minutes, even before the lecture was over, he interrupted to announce that he had the solution.}''
\vskip 2mm


When once asked about their famous result, Mark Kac replied the following (see \cite{C1986} and \cite{MK1985}): 
\vskip 2mm
``\emph{It took what looks now like a  miraculous confluence of circumstances to produce our result\dots. It would not have been enough, certainly not in 1939, to bring a  number theorist and a  probabilist together. It had to be Erd{\H o}s and me: Erd{\H o}s because he was almost unique in his knowledge and understanding of the number theoretic method of Viggo Brun,... and me because I could see independence and the normal law through the eyes of Steinhaus.}'' 
\vskip 2mm

We will now formulate the central limit theorem of Erd{\H o}s and Kac.

\begin{thm}[Erd\H{o}s-Kac, 1940]\label{thm:erdos kac}
Let $a,b\in\R$ with $a<b$. Then 
\[
\lim_{N\to\infty} \frac{1}{N}\bigg|\bigg\{ n\in\{1,\dots, N\}\,:\, a \leq \frac{\omega(n) - \ln\ln N}{\sqrt{\ln\ln N}} \leq b \bigg\} \bigg| = \frac{1}{\sqrt{2\pi}}\int_a^b e^{-x^2/2}\,\dint x\,.
\] 
\end{thm}

In other words, for large $N\in\N$ the proportion of natural numbers in the set $\{1,\dots,N\}$ for which the suitably normalized number of different prime factors is between $a$ and $b$ is close to a Gaussian integral from $a$ to $b$. In short: the number of prime factors of a large, suitably normalized number follow a Gaussian curve.

Providing a formal proof for Theorem \ref{thm:erdos kac} would go beyond the scope of this paper. The original argument of Erd{\H o}s and Kac use number theoretic methods of sieve theory (more precisely Brun's sieve). Another proof is due to Alfr\'ed R\'enyi (20. March 1921 in Budapest; 1. February 1970 Budapest) and P\'al Tur\'an (18. August 1910 in Budapest; 26. September 1976 Budapest) and can be found in \cite{RT1958}. Let us mention that Godfrey Harold Hardy (7. February 1877 in Cranleigh; 1. December 1947 in Cambridge) and Srinivasa Ramanujan (22. December 1887 in Erode; 26. April 1920 in Kumbakonam) prove in their paper \cite{HR1917} from 1917 that for all $\varepsilon\in(0,\infty)$
\[
\lim_{N\to\infty} \frac{1}{N}\bigg|\bigg\{ n\in\{1,\dots, N\}\,:\, \Big| \frac{\omega(n)}{\ln\ln N} -1\Big| \geq \varepsilon \bigg\} \bigg| = 0\,.
\] 
This means that for large $N\in\N$ if we pick a number $n\in\{1,\dots,N\}$ at random (with respect to the uniform distribution), then the number $\omega(n)$ of different prime factors is of order $\ln \ln N$.

\begin{rem}
Even though P\'al Tur\'an already noticed that the result of Hardy and Ramanujan can be obtained from an inequality for the second moments of $\omega(n)$ together with an application of Chebychev's inequality \cite{BD2010}, one can say that the Erd{\H o}s-Kac Theorem marks the beginning of probabilistic number theory. Also the work \cite{EW1939} of Paul Erd{\H o}s and Aurel Wintner (8. April 1903 in Budapest; 15. January 1958 in Baltimore) has been one of the pioneering contributions to this complex of problems.       
\end{rem}

We close this section with the statement of a corollary that gives a different
version of the Erd{\H o}s-Kac theorem, in which $N$ in the $\log\log$ terms
is replaced by $n$, which looks more natural in our setup, because it 
directly states that the distribution function of the sequence 
$\big(\frac{\omega(n) - \ln\ln n}{\sqrt{\ln\ln n}}\Big)_{n\in \N}$
 is that of the standard normal one.

\begin{cor} 
Let $a,b\in\R$ with $a<b$. Then
\[
\lim_{N\to\infty} \frac{1}{N}\bigg|\bigg\{ n\in\{1,\dots, N\}\,:\, a \leq \frac{\omega(n) - \ln\ln n}{\sqrt{\ln\ln n}} \leq b \bigg\} \bigg| = \frac{1}{\sqrt{2\pi}}\int_a^b e^{-x^2/2}\,\dint x\,.
\] 
\end{cor}

\begin{proof}
Clearly, for every $b\in\R$, we have 
\begin{align*}
\lefteqn{\limsup_{N\to\infty} \frac{1}{N}\bigg|\bigg\{ n\in\{1,\dots, N\}\,:\,  \frac{\omega(n) - \ln\ln n}{\sqrt{\ln\ln n}} \leq b \bigg\} \bigg| }\\
&\le\lim_{N\to\infty} \frac{1}{N}\bigg|\bigg\{ n\in\{1,\dots, N\}\,:\,  \frac{\omega(n) - \ln\ln N}{\sqrt{\ln\ln N}} \leq b \bigg\} \bigg| = \Phi(b)\,,
\end{align*}
where $\Phi(t)=\frac{1}{\sqrt{2\pi}}\int_{-\infty}^t
e^{-x^2/2}\,\dint x$ for all $t\in\R$ as before.
First note that, by Theorem \ref{thm:erdos kac}, the distribution functions $F_N$ with 
$F_N(t):=\tfrac{1}{N}|\{1\le n\le N\colon \omega(n)\le t\sqrt{\ln\ln N}+\ln\ln
N\}|$ converge pointwise to $\Phi$, and therefore also uniformly on $\R$, by 
Lemma \ref{lem:uniform-conv}. 

Now fix $b\in\R$ and let $K\in (0,\infty)$ be such that 
$e^{-\frac{K}{2}}<\tfrac{\varepsilon}{3}$.
Let $N_0\in\N$ be such that for all $N\ge N_0$
and all $t\in\R$ we have $F_N(t)\in(\Phi(t)-\tfrac{\varepsilon}{3},\Phi(t)+\tfrac{\varepsilon}{3})$, $\Phi\big(b-\tfrac{K}{\ln\ln N}\big)>\Phi(b)-\tfrac{\varepsilon}{3}$, $\sqrt{\ln\ln N}>b$, and $\ln \ln N>0$. 
With this 
\begin{align*}
\tfrac{1}{N}\big|\big\{ n\in\{1,\dots, N\}\colon \omega(n) \le b \sqrt{\ln\ln N}+\ln\ln N - K \big\} \big|\ge \Phi\big(b-\tfrac{K}{\ln\ln N}\big)-\tfrac{\varepsilon}{3}>\Phi\big(b\big)-\tfrac{2\varepsilon}{3}\,.
\end{align*}
If we denote $N_1:=\sup\{n\in\N\colon b \sqrt{\ln\ln N}+\ln\ln N - K > b \sqrt{\ln\ln n}+\ln\ln n\}$, then 
\begin{align*}
\lefteqn{\tfrac{1}{N}\big|\big\{ n\in\{1,\dots, N\}\colon \omega(n) \le b \sqrt{\ln\ln n}+\ln\ln n \big\} \big|}\\
&\ge \tfrac{1}{N}\big|\big\{ n\in\{N_1+1,\dots, N\}\colon \omega(n) \le b \sqrt{\ln\ln n}+\ln\ln n \big\} \big|\\
&\ge\tfrac{1}{N}\big|\big\{ n\in\{N_1+1,\dots, N\}\colon \omega(n) \le b \sqrt{\ln\ln N}+\ln\ln N - K \big\} \big|\\
&\ge\tfrac{1}{N}\big|\big\{ n\in\{1,\dots, N\}\colon \omega(n) \le b \sqrt{\ln\ln N}+\ln\ln N - K \big\} \big|-\tfrac{N_1}{N}\\
&> \Phi(b)-\tfrac{2\varepsilon}{3}-\tfrac{N_1}{N}\,.
\end{align*}
Now, let us compare $N$ and $N_1$. We observe that if
\[
b (\sqrt{\ln\ln N}-\sqrt{\ln\ln N_1})+\ln\ln N  -\ln\ln N_1 > K\,,
\]
then
\[
(\sqrt{\ln\ln N}+\sqrt{\ln\ln N_1})(\sqrt{\ln\ln N}-\sqrt{\ln\ln N_1})+\ln\ln N  -\ln\ln N_1 > K\,,
\]
which implies that
\[
2(\ln\ln N  -\ln\ln N_1) > K\,.
\]
Hence, we have
\[
\ln\ln N  - \tfrac{K}{2}>\ln\ln N_1
\]
and so $N^{e^{-\frac{K}{2}}}  > N_1$. Therefore, 
\begin{align*}
\frac{N_1}{N}<N^{e^{-\frac{K}{2}}-1}<N^{-K/2}<e^{-K/2}<\tfrac{\varepsilon}{3}\,,
\end{align*}
which completes the proof.
\end{proof}

A similar calculation shows that the two formulations of the Erd{\H o}s-Kac
theorem are actually equivalent.

\section{Some complementary considerations --- The case of lacunary series}

What we have seen so far shows the power of the concept of relative measure in number theory and how it can naturally (in large parts along the lines of classical probability theory) lead us to central limit theorems for number theoretic quantities, even where the axiomatic framework of Kolmogorov is not applicable. On the other hand, we have seen, when studying binary expansions, that Kolmogorov's theory is a powerful tool as well and allows us to obtain information about the Gaussian fluctuations of number theoretic quantities. A common spirit of both, and eventually a key to a Gaussian law, has always been a notion of independence. 

In what follows, we complement the previous considerations by showing that lacunary series, for instance those that are formed with functions $\cos(2\pi n_k \cdot):[0,1]\to\R$ and quickly increasing gap sequence $(n_k)_{k\in\N}$, behave in many ways like \emph{independent} random variables, and that this almost-independence or weak form of independence may still lead to fascinating results within the axiomatic theory of Kolmogorov.

Already in Subsection \ref{sec:dyadisch} on binary expansions we noted that Hans Rademacher introduced in \cite{Rademacher1922} what is known today as Rademacher functions. Those functions are defined in the following way,
\[
r_k(t) = \sign \sin(2^k\pi t),\qquad t\in[0,1],\, k\in\N\,,
\]
where for $x\in\R$,
\[
\sign(x) := 
\begin{cases}
-1 & : x<0\\
0 & :x=0\\
+1 & :x>0.
\end{cases}
\]
Rademacher studied the convergence behavior of series
\begin{align}\label{eq:rademacher series}
\sum_{k=1}^\infty a_kr_k(t),\qquad t\in[0,1],\,(a_k)_{k=1}^\infty \in\R^{\N}\,,
\end{align}
and proved that such series converge for almost all $t\in[0,1]$ if
\begin{align}\label{eq:a_k square summable}
\sum_{k=1}^\infty a_k^2 &< +\infty\,.
\end{align}
The necessity of square integrability was obtained by Alexander Khintchine (19. July 1894 in Kondyrjowo; 18. November 1959 in Moscow) and Andrei Kolmogorov in their 1925 paper \cite{KK1925}, showing that if
\begin{align}\label{eq:a_k not square summable}
\sum_{k=1}^\infty a_k^2 = +\infty,
\end{align}
then the series \eqref{eq:rademacher series} diverges for almost all $t\in[0,1]$.

Starting in the 1920s, Stefan Banach (30. March 1892 in Krakow; 31. August 1945 in Lviv), Andrei Kolmogorov, Raymond Paley (7. January 1907 in Bournemouth; 7. April 1933 near Banff), Antoni Zygmund (25. December 1900 in Warsaw; 30. May 1992 in Chicago) and others studied the convergence behavior of trigonometric series
\begin{align}\label{eq:lacunary series}
\sum_{k=1}^\infty a_k \cos(2\pi n_k t),\qquad t\in[0,1],\,(a_k)_{k=1}^\infty\in\R^{\N}\,,
\end{align}
where the sequence $(n_k)_{k=1}^\infty$ satisfies the Hadamard gap condition
\[
\frac{n_{k+1}}{n_k}>q>1
\]
for all $k\in\N$ (see \cite{B1930, K1924, PZ1930, Z1930}). For such series one can obtain results similar to those for Rademacher series \eqref{eq:rademacher series}. Kolmogorov could prove in \cite{K1924} that the square summability condition \eqref{eq:a_k square summable} is also sufficient for almost everywhere convergence of lacunary series. The necessity of \eqref{eq:a_k square summable} has been shown by Zygmund in \cite{Z1930}.

An important analogy between Rademacher series and lacunary series, in particular in view of our article, remained unnoticed for a long time. In Subsection \ref{sec:dyadisch} we proved that the Rademacher functions (more precisely a version of them) are independent. In particular, given any sequence $(a_k)_{k=1}^\infty$ of real numbers, the functions $a_kr_k$, $k\in\N$ are independent (but no longer identically distributed), and we have for all $k\in\N$ that
\[
\E[a_kr_k]= 0 \qquad\text{and}\qquad \Var[a_kr_k] = a_k^2\,.
\]
Using the notation from Lindeberg's theorem (see Theorem \ref{thm:lindeberg}), we see that
\[
s_n^2 = \sum_{k=1}^n \Var[a_kr_k] = \sum_{k=1}^n a_k^2\,.
\]  
But this means that for $\varepsilon\in(0,\infty)$, Lindeberg's condition for the weighted Rademacher functions reads as follows,
\[
\frac{1}{\sum\limits_{k=1}^n a_k^2} \,\sum_{k=1}^n \E\Bigg[ (a_kr_k)^2 \ind_{\Big\{|a_kr_k| \geq \varepsilon \sqrt{\sum_{k=1}^n a_k^2} \Big\}} \Bigg]
= \frac{1}{\sum\limits_{k=1}^n a_k^2} \,\sum_{k=1}^n a_k^2\, \Pro\Bigg[ |a_k| \geq \varepsilon \sqrt{\sum_{k=1}^n a_k^2}\, \Bigg]\,.
\]
For Lindeberg's condition to be satisfied, we require the right-hand side to converge to $0$ as $n\to\infty$. A moment's thought, however, reveals that this is the case whenever
\begin{eqnarray}\label{eq:Lindeberg requirements a_k}
\sum_{k=1}^\infty a_k^2 = +\infty &\qquad\text{and}\qquad& \max_{1\leq k \leq n} |a_k| = o\Bigg( \sqrt{\sum_{k=1}^n a_k^2}\Bigg)\,.
\end{eqnarray}
Therefore, under condition \eqref{eq:Lindeberg requirements a_k}, we obtain that, for all $t\in\R$,
\[
\lim_{n\to\infty} \lambda\Bigg( \bigg\{x\in[0,1]\,:\, \sum_{k=1}^n a_kr_k(x) \leq t \sqrt{\sum_{k=1}^na_k^2}\,  \bigg\} \Bigg) = \frac{1}{\sqrt{2\pi}}\int_{-\infty}^{t} e^{-\frac{y^2}{2}}\,\dint y\,.
\]
It was not before 1947 that Rapha\"el Salem (7. November 1898 in Saloniki; 20. June 1963 in Paris) and Antoni Zygmund proved in \cite{SZ1947} that for Hadamard gap sequences the functions $\big(\cos(2\pi n_k \cdot)\big)_{k\in\N}$ follow a central limit theorem, i.e., for all $t\in\R$,
\[
\lim_{N\to\infty} \lambda\bigg(\Big\{ x\in(0,1) \,:\, \sum_{k=1}^N \cos(2\pi n_k x) \leq t \sqrt{N/2}\Big\} \bigg) = \frac{1}{\sqrt{2\pi}}\int_{-\infty}^{t} e^{-\frac{y^2}{2}}\,\dint y\,.
\]
For sequences with very large gaps, i.e., those satisfying the stronger condition
\[
\frac{n_{k+1}}{n_k} \stackrel{n\to\infty}{\longrightarrow} +\infty\,,
\] 
such a central limit theorem had been obtained in 1939 by Mark Kac in \cite{K1939}.

Around the same time as Salem and Zygmund, Mark Kac \cite{K1946} (see also \cite{K1949, MK1959} and the references therein) obtained a central limit theorem for functions $f:\R\to\R$ of bounded variation on $[0,1]$ satisfying
\[
f(t+1) = f(t) \qquad\text{and}\qquad \int_0^1f(t) \,\dint t = 0\,.
\]
He showed that for such functions
\[
\lim_{N\to\infty} \lambda\bigg(\Big\{ x\in(0,1) \,:\, \sum_{k=1}^N f(2^kx) \leq t \sigma \sqrt{N} \Big\} \bigg) = \frac{1}{\sqrt{2\pi}}\int_{-\infty}^{t} e^{-\frac{y^2}{2}}\,\dint y\,
\]
whenever
\begin{align}\label{eq:sigma^2}
\sigma^2 := \int_0^1 f(t)^2\,\dint t + 2 \sum_{k=1}^\infty \int_0^1 f(t) f(2^kt)\,\dint t \neq 0\,. 
\end{align}
This already indicates that the functions $f(2^k\cdot)$, $k\in\N$ do not behave like independent random variables. In fact, in that case we would expect something like
\[
\sigma^2 = \int_0^1 f(t)^2 \,\dint t \neq 0 
\] 
rather than condition \eqref{eq:sigma^2}. After further progress had been made by Gapo\v{s}kin \cite{G1966} and Takahashi \cite{T1961}, Gapo\v{s}kin eventually discovered a deep connection between the validity of a central limit theorem and the number of solutions of a certain Diophantine equation \cite{G1970}, i.e., whether a central limit theorem holds or not depends not only on the growth rate of the sequence $(n_k)_{k\in\N}$, but also critically on its number theoretic properties. In 2010 Christoph Aistleitner and Istv\'an Berkes presented a paper in which they obtained both necessary and sufficient conditions under which a sequence $f(n_k \cdot)_{k\in\N}$ follows a Gaussian law of errors \cite{AB2010}.

Please note that the preceding paragraph is not intended to be exhaustive. Still it indicates the development of the subject, highlights some fascinating results, and shows how analytic, probabilistic, and number theoretic arguments and properties intertwine.

\begin{rem}
The results presented in this final section are not restricted to central limit phenomena. Beyond the normal fluctuations one can also prove laws of the iterated logarithm for lacunary series and we refer the reader to the work of Erd{\H o}s and G\'{a}l \cite{erdosgal1955}, Aistleitner and Fukuyama \cite{AF2012, AF2016},  Aistleitner, Berkes, and Tichy \cite{ABT2012a, ABT2012}, and the references cited therein.
\end{rem}
\subsection*{Acknowledgment}
GL is supported by the Austrian Science Fund (FWF) Project F5508-N26, which is part of the Special Research Program ``Quasi-Monte Carlo Methods: Theory and Applications''. JP is supported by the Austrian Science Fund (FWF) Project P32405 ``Asymptotic Geometric Analysis and Applications'' as well as a visiting professorship from Ruhr University Bochum and its Research School PLUS.

\bibliographystyle{plain}

\end{document}

%% file: gausskurve.pdf_tex
\begingroup%
  \makeatletter%
  \providecommand\color[2][]{%
    \errmessage{(Inkscape) Color is used for the text in Inkscape, but the package 'color.sty' is not loaded}%
    \renewcommand\color[2][]{}%
  }%
  \providecommand\transparent[1]{%
    \errmessage{(Inkscape) Transparency is used (non-zero) for the text in Inkscape, but the package 'transparent.sty' is not loaded}%
    \renewcommand\transparent[1]{}%
  }%
  \providecommand\rotatebox[2]{#2}%
  \ifx\svgwidth\undefined%
    \setlength{\unitlength}{496.7608887bp}%
    \ifx\svgscale\undefined%
      \relax%
    \else%
      \setlength{\unitlength}{\unitlength * \real{\svgscale}}%
    \fi%
  \else%
    \setlength{\unitlength}{\svgwidth}%
  \fi%
  \global\let\svgwidth\undefined%
  \global\let\svgscale\undefined%
  \makeatother%
  \begin{picture}(1,0.29195133)%
    \put(0.1,0){\includegraphics[width=\unitlength,page=1]{gausskurve.pdf}}%
    \put(0.10469037,0.00748851){\color[rgb]{0,0,0}\makebox(0,0)[lb]{\smash{-}}}%
    \put(0.1152528,0.00748851){\color[rgb]{0,0,0}\makebox(0,0)[lb]{\smash{3}}}%
    \put(0.1,0){\includegraphics[width=\unitlength,page=2]{gausskurve.pdf}}%
    \put(0.22048049,0.00748851){\color[rgb]{0,0,0}\makebox(0,0)[lb]{\smash{-}}}%
    \put(0.23104292,0.00748851){\color[rgb]{0,0,0}\makebox(0,0)[lb]{\smash{2}}}%
    \put(0.1,0){\includegraphics[width=\unitlength,page=3]{gausskurve.pdf}}%
    \put(0.33627061,0.00748851){\color[rgb]{0,0,0}\makebox(0,0)[lb]{\smash{-}}}%
    \put(0.34683304,0.00748851){\color[rgb]{0,0,0}\makebox(0,0)[lb]{\smash{1}}}%
    \put(0.1,0){\includegraphics[width=\unitlength,page=4]{gausskurve.pdf}}%
    \put(0.57313507,0.00748851){\color[rgb]{0,0,0}\makebox(0,0)[lb]{\smash{1}}}%
    \put(0.1,0){\includegraphics[width=\unitlength,page=5]{gausskurve.pdf}}%
    \put(0.68892519,0.00748851){\color[rgb]{0,0,0}\makebox(0,0)[lb]{\smash{2}}}%
    \put(0.1,0){\includegraphics[width=\unitlength,page=6]{gausskurve.pdf}}%
    \put(0.80471532,0.00748851){\color[rgb]{0,0,0}\makebox(0,0)[lb]{\smash{3}}}%
    \put(0.1,0){\includegraphics[width=\unitlength,page=7]{gausskurve.pdf}}%
    \put(0.43117543,0.0823686){\color[rgb]{0,0,0}\makebox(0,0)[lb]{\smash{0.1}}}%
    \put(0.1,0){\includegraphics[width=\unitlength,page=8]{gausskurve.pdf}}%
    \put(0.43117543,0.14026366){\color[rgb]{0,0,0}\makebox(0,0)[lb]{\smash{0.2}}}%
    \put(0.1,0){\includegraphics[width=\unitlength,page=9]{gausskurve.pdf}}%
    \put(0.43117543,0.19815872){\color[rgb]{0,0,0}\makebox(0,0)[lb]{\smash{0.3}}}%
    \put(0.1,0){\includegraphics[width=\unitlength,page=10]{gausskurve.pdf}}%
    \put(0.433117543,0.25605377){\color[rgb]{0,0,0}\makebox(0,0)[lb]{\smash{0.4}}}%
    \put(0.1,0){\includegraphics[width=\unitlength,page=11]{gausskurve.pdf}}%
    \put(0.5881352,0.23082565){\color[rgb]{0,0,0}\makebox(0,0)[lb]{\smash{\textbf{$x\mapsto \frac{e^{-\frac{x^2}{2}}}{\sqrt{2\pi}}$}}}}%
  \end{picture}%
\endgroup%